\documentclass{amsart}
\usepackage{amssymb, amsmath, latexsym, mathabx, dsfont, enumerate, multirow}

\renewcommand{\baselinestretch}{\baselinestretch}
\renewcommand{\baselinestretch}{1.1}
\numberwithin{equation}{section}

\newtheorem{thm}{Theorem}[section]

\newtheorem{conj}[thm]{Conjecture}

\newcommand{\Mod}[1]{\ (\mathrm{mod}\ #1)}

\begin{document}

\title{Almost universal sums of triangular numbers with one exception}
\author{Jangwon Ju}

\address{Department of Mathematics, University of Ulsan, Ulsan 44610, Republic of Korea}
\email{jangwonju@ulsan.ac.kr}



\begin{abstract}
For an arbitrary integer $x$, an integer of the form $T(x)=\frac{x^2+x}{2}$ is called a triangular number.
For positive integers $\alpha_1,\alpha_2,\dots,\alpha_k$, a sum $\Delta_{\alpha_1,\alpha_2,\dots,\alpha_k}(x_1,x_2,\dots,x_k)=\alpha_1 T(x_1)+\alpha_2 T(x_2)+\cdots+\alpha_k T(x_k)$ of triangular numbers is said to be {\it almost universal  with one exception} if the Diophantine equation $\Delta_{\alpha_1,\alpha_2,\dots,\alpha_k}(x_1,x_2,\dots,x_k)=n$ has an integer solution $(x_1,x_2,\dots,x_k)\in\mathbb{Z}^k$ for any nonnegative integer $n$ except a single one.
In this article, we classify all almost universal sums of triangular numbers with one exception. 
Furthermore, we provide an effective criterion on almost universality with one exception of an arbitrary sum of triangular numbers, which is a generalization of ``15-theorem" of Conway, Miller and Schneeberger.
\end{abstract}

\maketitle

\section{Introduction}

In 1770, Lagrange proved that every nonnegative integer can be written as a sum of at most four squares of integers. 
Motivated by Lagrange's four-square theorem, Ramanujan provided a list of $55$ candidates of diagonal quaternary integral quadratic forms that represent all nonnegative integers (for details, see \cite{rama}). 
 Dickson pointed out that the diagonal quaternary quadratic form $x^2+2y^2+5z^2+5t^2$ in Ramanujan's list doesn't represent the integer $15$, and confirmed that Ramanujan's assertion is correct for all the other $54$ quadratic forms (for details, see \cite{dick}).  
 
Ramanujan's assertion was generalized to find all {\it universal} quaternary quadratic forms, i.e., those representing all nonnegative integers. 
This was completely solved by Conway, Miller, and Schneeberger in 1993. 
They proved the so called ``15-theorem", which states that a positive definite integral quadratic form is universal if and only if it represents the integers 
$$
1,\ 2,\ 3,\ 5,\ 6,\ 7,\ 10,\ 14, \ \text{and} \ 15,
$$ 
irrespective of its rank.
Moreover, they provided a complete list of 204 quaternary quadratic forms with this property. 
Recently, Bhargava provided an elegant proof of the 15-Theorem in \cite{B}.

As a natural generalization of the 15-theorem,  Bhargava and Hanke \cite{BH} proved the so-called ``290-theorem", which states that    
every positive definite integer-valued quadratic form is universal if and only if it represents the integers 
$$
\begin{array} {ll}
1,\ 2, \ 3,\ 5,\ 6,\ 7,\ 10,\ 13,\ 14,\ 15,\ 17,\ 19,\ 21,\ 22,\ 23,\ 26,\ 29,\\
30,\ 31,\ 34,\ 35,\ 37,\ 42,\ 58,\ 93,\ 110,\ 145,\ 203, \ \text{and} \ 290.
\end{array}
$$
Here a quadratic form $f(x_1,x_2,\dots,x_n)=\sum_{1 \le i, j\le n} a_{ij} x_ix_j \ (a_{ij}=a_{ji})$ 
is called {\it integral} if $a_{ij}\in \mathbb Z$ for any $i,j$, and is called {\it integer-valued} if $a_{ii}\in \mathbb Z$ and $a_{ij}+a_{ji}\in\mathbb Z$ for any $i,j$.
Moreover, they provided a complete list of 6436 such forms in four variables.

A next natural generalization of Ramanujan's assertion is to classify all quadratic forms representing all nonnegative integers with finitely many exceptions. A quadratic form with this property is said to be {\it almost universal}. 
At first, by using an escalation method, in \cite{hal} Halmos provided a list of 88 candidates of almost universal diagonal quaternary quadratic forms with one exception. 
It was pointed out that diagonal quaternary quadratic forms $x^2+y^2+2z^2+22t^2$ and $x^2+2y^2+4z^2+22t^2$ in Halmos's list don't represent two integers $14$ and $78$.
Halmos proved that 85 of those indeed represent all nonnegative integers except a single one.   
Moreover, he conjectured that the remaining form $x^2+2y^2+7z^2+13t^2$ represents all nonnegative integers except $5$, and it was proved by Pall in \cite{pall}.

In 2009, Bochnak and Oh \cite{BO} provided an effective characterizations for deciding whether a positive definite integral quaternary quadratic form represents all nonnegative integers with finitely many exceptions. 
It can be considered as the final solution to the problem first addressed by Ramanujan in \cite{rama}.

In this paper, we investigate representations of sums of triangular numbers.
The $n$-th triangular number is the number of dots in the triangular arrangement with $n$ dots on a side.
More precisely, the $n$-th triangular number is defined by 
$$
T(n)=\frac{n^2+n}{2}
$$ 
for any nonnegative integer $n$.
Note that $\{T(x) : x\in \mathbb{N}\cup\{0\}\}=\{T(x) : x\in\mathbb{Z}\}$.

For positive integers $\alpha_1,\alpha_2,\dots,\alpha_k$, we say a sum 
$$
\Delta_{\alpha_1,\alpha_2,\dots,\alpha_k}(x_1,x_2,\dots,x_k):=\alpha_1 T(x_1)+\alpha_2 T(x_2)+\cdots+\alpha_k T(x_k)
$$ 
of triangular numbers {\it represents} a nonnegative integer $n$ if the diophantine equation 
$$
\Delta_{\alpha_1,\alpha_2,\dots,\alpha_k}(x_1,x_2,\dots,x_k)=n
$$
has an integer solution $(x_1,x_2,\dots,x_k)\in\mathbb{Z}^k$. 
Furthermore, a sum
$$
\Delta_{\alpha_1,\alpha_2,\dots,\alpha_k}(x_1,x_2,\dots,x_k) \ (\text{simply}, \ \Delta_{\alpha_1,\alpha_2,\dots,\alpha_k}) 
$$
of triangular number is called {\it universal} if it represents all nonnegative integers.

The famous Gauss' triangular theorem states that every positive integer can be expressed as a sum of three triangular numbers which was first asserted by Fermat in 1638. 
In 1862, Liouville proved that for positive integers $a,b$, and $c~ (a\leq b\leq c)$, a sum  $\Delta_{a,b,c}$ of triangular numbers is universal if and only if $(a,b,c)$ is one of the following triples:
$$
(1,1,1),\quad (1,1,2),\quad (1,1,4),\quad (1,1,5),\quad (1,2,2),\quad (1,2,3),\quad\text{and}\quad (1,2,4),
$$
which is a generalization of Gauss' triangular theorem.

In 2013, Bosma and Kane proved the triangular theorem of eight which states that for positive integers $\alpha_1,\alpha_2,\dots,\alpha_k$, an arbitrary sum $\Delta_{\alpha_1,\alpha_2,\dots,\alpha_k}$ of triangular numbers is universal if and only if it represents $1,2,4,5$, and $8$ (for details, see \cite{BK}). 
This might be considered as a natural generalization of the ``15-theorem'' of Conway, Miller, and Schneeberger.

For positive integers $\alpha_1,\alpha_2,\dots,\alpha_k$, a sum $\Delta_{\alpha_1,\alpha_2,\dots,\alpha_k}$ of triangular numbers is called {\it almost universal} if it represents all nonnegative integers with finitely many exceptions.
Especially, if a sum of triangular numbers represents all nonnegative integers except a single one, then it is said to be {\it almost universal with one exception}.
Furthermore, it is called {\it proper} if any proper partial sum of it doesn't represent at least two nonnegative integers.

We know that if a sum $\Delta_{\alpha_1,\alpha_2,\dots,\alpha_k}$ of triangular numbers is almost universal with one exception $m$, then $m$ is inside $\{1,2,4,5,8\}$ by the triangular theorem of eight.
By using an escalation method, we give a complete list of candidates of 490 proper almost universal sums of triangular numbers with one exception, actually, the numbers of ternary, quaternary, and quinary sums among them are 1, 235, and 254, respectively. 
We classify all almost universal sums of triangular numbers with one exception $1,2,4,5$, and $8$, respectively.
Furthermore, we provide an effective criterion on almost universality with one exception of an arbitrary sum  $\Delta_{\alpha_1,\alpha_2,\dots,\alpha_k}$ of triangular numbers. This might be considered as a natural generalization of the 15-theorem of Conway, Miller, and Schneeberger.

\begin{thm}\label{1}
A sum of triangular numbers is almost universal with one exception $1$ if and only if it represents the integers
$$
2,~3,~4,~8,~10,~16,~and~19
$$
and doesn't represent $1$.
There are exactly 29 proper almost universal sums of triangular numbers with one exception $1$, actually, there are 11 quaternary and 18 quinary ones (see Table 1). 
\end{thm}

\begin{thm}\label{2}
A sum of triangular numbers is almost universal with one exception $4$ if and only if it represents the integers
$$
1,~2,~11,~14,~19,~25,~29,~46,~ and ~50
$$
and doesn't represent $4$.
There are exactly 138 proper almost universal sums of triangular numbers with one exception $4$, actually, there are 127 quaternary and 11 quinary ones (see Table 2). 
\end{thm}

\begin{thm}\label{3}
A sum of triangular numbers is almost universal with one exception $5$ if and only if it represents the integers
$$
1,~2,~8,~14,~26,~40,~41,~47,~59,~ and ~71
$$
and doesn't represent $5$.
There are exactly 171 proper almost universal sums of triangular numbers with one exception $5$, actually, there are 56 quaternary and 115 quinary ones (see Table 10). 
\end{thm}

\begin{thm}\label{4}
A sum of triangular numbers is almost universal with one exception $8$ if and only if it represents the integers
$$
1,~2,~5,~17,~ and~ 89
$$
and doesn't represent $8$.
There are exactly 80 proper almost universal sums of triangular numbers with one exception $8$, actually, there are 7 quaternary and 73 quinary ones (see Table 13). 
\end{thm}

Note that the sum $\Delta_{1,4,5}$ of triangular numbers is the unique candidate of ternary almost universal sums of triangular numbers with one exception (see Section 7). 
In \cite{Kane}, Kane proved that $\Delta_{1,4,5}$ represents all positive odd integers under the assuming GRH for $L$-functions of weight $2$ newforms. 
We conjecture that it represents all nonnegative integers except $2$.
Actually, we checked that $\Delta_{1,4,5}$ represents all nonnegative integers up to $10^7$ except $2$.

\begin{conj}\label{Conj}
The ternary sum $\Delta_{1,4,5}$ of triangular numbers is almost universal with one exception $2$.
\end{conj} 
\noindent Assume that Conjecture \ref{Conj} is true. Then we have the following theorem.
\begin{thm}\label{5}
 A sum of triangular numbers is almost universal with one exception $2$ if and only if it represents the integers
$$
1,~4,~5,~7,~8,~9,~11,~16,~17,~20,~29,~ and ~35
$$
and doesn't represent $2$.
There are exactly 72 proper almost universal sums of triangular numbers with one exception $2$, actually, there are unique ternary, 34 quaternary, and 37 quinary ones (see Table 15).
\end{thm}

The complete list of almost universal sums of triangular numbers with one exception is given in Tables 1, 2, 10, 13, and 15.
In the above tables, each sum of triangular numbers having a dagger mark with the last coefficient is almost universal with one exception that is not proper. 
In Table 13, each sum of triangular numbers having an asterisk with the last coefficient is universal.

Let
$\displaystyle
f(x_1,x_2,\dots,x_k)=\sum_{1 \le i, j\le k} a_{ij} x_ix_j \ (a_{ij}=a_{ji} \in \mathbb Z)
$
 be a positive definite integral quadratic form. 
The corresponding integral symmetric matrix of $f$ is defined by $M_f=(a_{ij})$ and any matrix isometric to it is denoted by $M_f$ also. 
For a diagonal quadratic form $f(x_1,x_2,\dots,x_k)=a_1x_1^2+a_2x_2^2+\cdots+a_kx_k^2$, we simply write 
$$
M_f=\langle a_1,a_2,\dots,a_k\rangle.
$$ 
For an integer $n$, we say $n$ is {\it represented by $f$} if the equation $f(x_1,x_2,\dots,x_k)=n$ has an integer solution $(x_1,x_2,\dots,x_k)\in\mathbb{Z}^k$, which is denoted by $n\longrightarrow f$. 
The genus of $f$, denoted by $\text{gen}(f)$, is the set of all quadratic forms that are isometric to $f$ over the $p$-adic integer ring $\mathbb{Z}_p$ for any prime $p$. 
The number of isometry classes in $\text{gen}(f)$ is called the class number of $f$ and denoted by $h(f)$.

A good introduction to the theory of quadratic forms may be found in \cite{om}, and we adopt the notations and terminologies from this book.


\section{General tools}

For positive integers $\alpha_1,\alpha_2,\dots,\alpha_k$, we define
$$
\Delta_{\alpha_1,\alpha_2,\dots,\alpha_k}(x_1,x_2,\dots,x_k)=\alpha_1T(x_1)+\alpha_2T(x_2)+\cdots+\alpha_kT(x_k).
$$
Recall that a sum 
$$
\Delta_{\alpha_1,\alpha_2,\dots,\alpha_k}(x_1,x_2,\dots,x_k) \ (\text{simply}, \ \Delta_{\alpha_1,\alpha_2,\dots,\alpha_k})
$$ 
of triangular numbers is called almost universal if it represents all nonnegative integers with finitely many exceptions. In particular, if the number of exceptions is one, then it is said to be almost universal with one exception, which is equivalent to the existence of an integer solution $(x_1,x_2,\dots,x_k)\in\mathbb{Z}^k$ of 
$$
\alpha_1(2x_1+1)^2+\alpha_2(2x_2+1)^2+\cdots+\alpha_k(2x_k+1)^2=8n+\alpha_1+\alpha_2\cdots+\alpha_k 
$$
for any nonnegative integer $n$ except a single one.
Furthermore, this is equivalent to the existence of an integer solution $(x_1,x_2,\dots,x_k)\in\mathbb{Z}^k$ of  
\begin{equation}\label{congruence condition}
\alpha_1x_1^2+\alpha_2x_2^2+\cdots+\alpha_kx_k^2=8n+\alpha_1+\alpha_2+\cdots+\alpha_k~\text{with}~x_1x_2\cdots x_k\equiv1\Mod2
\end{equation}
for any nonnegative integer $n$ except a single one.

Now, we introduce our strategy to prove that a sum $\Delta_{\alpha_1,\alpha_2,\dots,\alpha_k}$ of triangular numbers is almost universal with one exception.
At first, take a suitable ternary section $\Delta_{\alpha_{i_1},\alpha_{i_2},\alpha_{i_3}}$ of $\Delta_{\alpha_1,\alpha_2,\dots,\alpha_k}$, where $\{\alpha_{i_1},\alpha_{i_2},\alpha_{i_3}\} \subset \{\alpha_1,\alpha_2, \cdots ,\alpha_k\}$. 
Without loss of generality, we may assume that $\Delta_{\alpha_{i_1},\alpha_{i_2},\alpha_{i_3}}=\Delta_{\alpha_1,\alpha_2,\alpha_3}$.
We consider the equation
\begin{equation}\label{ternary form}
\alpha_1 x_1^2+\alpha_2 x_2^2+\alpha_3 x_3^2=8n+\alpha_1+\alpha_2+\alpha_3 ~\text{with}~ x_1x_2x_3\equiv1\Mod2.
\end{equation}
Note that Equation \eqref{ternary form} corresponds to the representations by a ternary quadratic form with congruence conditions. 
Since there are some method for determining the existence of representations of integers by a ternary quadratic form, we try to find a suitable method on reducing Equation \eqref{ternary form} to the representations of a ternary quadratic form, denoted by $f(x_1,x_2,x_3)$, without congruence conditions.
To explain our method, for example, assume that $\alpha_1\equiv\alpha_2\equiv0\Mod2$ and $\alpha_3 \equiv1\Mod2$. 
Then Equation \eqref{ternary form} has an integer solution if 
\begin{equation}\label{without congruence condition}
f(x_1,x_2,x_3)=\alpha_1(x_3-2x_1)^2+\alpha_{2}(x_3-2x_{2})^2+\alpha_3x_3^2=8n+\alpha_1+\alpha_2+\alpha_3
\end{equation}
has an integer solution. 
Hence, in this case, the problem can be reduced to the representations of a ternary quadratic form without congruence conditions.

After that for sufficiently large $n$, we find suitable $a_4,\dots,a_k\in\mathbb{Z}$ such that
\begin{equation}\label{conditions}
\begin{array}{rl}
\rm{(i)}   &  a_4\cdots a_k\equiv1\Mod2; \\
\rm{(ii)}  &  8n+\alpha_1+\alpha_2+\cdots+\alpha_k-(\alpha_4 a_4^2+\cdots+\alpha_k a_k^2)\geq0; \\
\rm{(iii)} &  8n+\alpha_1+\alpha_2+\cdots+\alpha_k-(\alpha_4 a_4^2+\cdots+\alpha_k a_k^2) \longrightarrow f(x_1,x_2,x_3). \\
\end{array}
\end{equation}
Then we know that Equation \eqref{congruence condition} has an integer solution.
Finally, we directly check that the sum $\Delta_{\alpha_1,\alpha_2,\dots,\alpha_k}$ of triangular numbers represents all remaining small integers except a single one.

In \cite{4-8}, \cite{regular}, and \cite{pentagonal}, we developed a method that determines whether or not integers in an arithmetic progression are represented by some particular ternary quadratic form. 
We briefly introduce this method for those who are unfamiliar with it. 
 
 Let $d$ be a positive integer and let $a$ be a nonnegative integer $(a\leq d)$. 
We define 
$$
S_{d,a}=\{dn+a \mid n \in \mathbb N \cup \{0\}\}.
$$
For two positive definite integral ternary quadratic forms $f,g$, we define
$$
R(g,d,a)=\{v \in (\mathbb{Z}/d\mathbb{Z})^3 \mid vM_gv^t\equiv a \ (\text{mod }d) \}
$$
and
$$
R(f,g,d)=\{T\in M_3(\mathbb{Z}) \mid  T^tM_fT=d^2M_g \}.
$$
Since $f$ and $g$ are positive definite, the above two sets are always finite.
A coset (or, a vector in the coset) $v \in R(g,d,a)$ is said to be {\it good} with respect to $f,g,d, \text{ and }a$ if there is a $T\in R(f,g,d)$ such that $\frac1d \cdot vT^t \in \mathbb{Z}^3$.  
The set of all good vectors in $R(g,d,a)$ is denoted by $R_f(g,d,a)$.   
If  $R(g,d,a)=R_f(g,d,a)$, we write  
$$
g\prec_{d,a} f.
$$ 

Now, we introduce two theorems which play a crucial role in proving our results.
\begin{thm}\label{good}
Under the same notations given above, if $g\prec_{d,a} f$, then 
$$
S_{d,a}\cap Q(g) \subset Q(f).
$$
\end{thm}
\begin{proof}
The theorem follows directly from Lemma 2.2 of \cite{regular} (see also Theorem 2.1 in \cite{Kaps}). 
\end{proof} 

\begin{thm}\label{pme}
Assume that  $T\in M_3(\mathbb Z)$ satisfies the following conditions:
\begin{enumerate}
\item[(i)] $\frac1dT$ has an infinite order;
\item[(ii)] $T^tM_gT=d^2M_g$;
\item[(iii)]  for any vector $v \in \mathbb Z^3$ such that  $v\, (\text{mod }d)\in B_f(g,d,a)$, $\frac1d\cdot vT^t\in \mathbb Z^3$. 
\end{enumerate}
Then we have
$$
S_{d,a} \cap Q(g)\setminus \{g(z)\cdot s^2\mid s\in\mathbb{Z}\}\subset Q(f),
$$
where the vector $z$ is any integral primitive eigenvector of $T$. 
\end{thm}
\begin{proof}
See Theorem 2.1 of \cite{4-8}.
\end{proof}
We define  $B_f(g,d,a)=R(g,d,a)\setminus R_f(g,d,a)$ and its cardinality is denoted by $|B_f(g,d,a)|~(simply, |B|)$.
In general, if $d$ is large, then it is hard to compute the set $B_f(g,d,a)$ exactly by hand. 
A MAGMA based  computer program for computing this set  is available upon request to the author.


\section{Proofs of Theorem \ref{1} }
We say an almost universal sum $\Delta_{\alpha_1,\alpha_2,\dots,\alpha_k}$ of triangular numbers with one exception is {\it proper} if for any proper subset $\{i_1,i_2,\dots,i_u\}\subset \{1,2,\dots,k\}$, the partial sum $\Delta_{\alpha_{i_l},\alpha_{i_2},\dots,\alpha_{i_u}}$ doesn't represent at least two nonnegative integers. 
When the sum $\Delta_{\alpha_1,\alpha_2,\dots,\alpha_k}$ of triangular numbers is not universal, the $n$-th nonnegative integers that is not represented by $\Delta_{\alpha_1,\alpha_2,\dots,\alpha_k}$ is called an $n$-th {\it truant} of $\Delta_{\alpha_1,\alpha_2,\dots,\alpha_k}$ and denoted by $\mathfrak{T}_n(\Delta_{\alpha_1,\alpha_2,\dots,\alpha_k})$ if it exists.

\begin{proof}[Proof of Theorem \ref{1}] Let $\alpha_1,\alpha_2,\dots,\alpha_k$ be positive integers. Assume that a sum $\Delta_{\alpha_1,\alpha_2,\dots,\alpha_k}$ of triangular numbers is almost universal with one exception $1$.
Without loss of generality, we may assume that $\alpha_1\leq\alpha_2\leq\cdots\leq\alpha_k$.
We know that $\alpha_1$ should be $2$ since it represents all nonnegative integers except 1.
Then $2\leq\alpha_2\leq3$ since $\mathfrak{T}_2(\Delta_2)=3$.
Note that 
$$
\mathfrak{T}_2(\Delta_{\alpha_1,\alpha_2})=
\begin{cases}
3\quad \text{if}~ (\alpha_1,\alpha_2)=(2,2),\\
4\quad\text{if}~ (\alpha_1,\alpha_2)=(2,3).
\end{cases}
$$
Therefore $(\alpha_1,\alpha_2,\alpha_3)$ should be one of the followings:
$$
(2,2,2),\quad(2,2,3),\quad(2,3,3),\quad\text{and}\quad(2,3,4).
$$
One may easily check that there are no ternary almost universal sums of triangular numbers with one exception $1$. 
Indeed, for each of the above four cases, the second truant is
$$
\mathfrak{T}_2(\Delta_{\alpha_1,\alpha_2,\alpha_3})=
\begin{cases}
{\setlength\arraycolsep{1pt}
\begin{array}{llll}
&3\quad &\text{if}&~(\alpha_1,\alpha_2,\alpha_3)=(2,2,2),\\ 
&10\quad &\text{if}&~(\alpha_1,\alpha_2,\alpha_3)=(2,2,3),\\
&4\quad &\text{if}&~(\alpha_1,\alpha_2,\alpha_3)=(2,3,3),\\
&8\quad &\text{if}&~(\alpha_1,\alpha_2,\alpha_3)=(2,3,4).
\end{array}}
\end{cases}
$$
Therefore, $\alpha_3\leq\alpha_4\leq\mathfrak{T}_2(\Delta_{\alpha_1,\alpha_2,\alpha_3})$ for each possible case. 
So there are 17 candidates of quaternary almost universal sums of triangular numbers with one exception 1. 
One may easily check that 6 sums of them don't represent at least two nonnegative integers.
Actually, we know that
\begin{equation}\label{truant}
\mathfrak{T}_2(\Delta_{\alpha_1,\alpha_2,\alpha_3,\alpha_4})=
\begin{cases}
{\setlength\arraycolsep{1pt}
\begin{array}{llllll}
&\mathfrak{T}_2(\Delta_{\alpha_1,\alpha_2,\alpha_3})\quad &\text{if}&~(\alpha_1,\alpha_2,\alpha_3,\alpha_4)&=&(2,2,2,2),(2,2,3,9),\\
&~&~&~&~&(2,3,3,3),(2,3,4,7),\\ 
&19\quad &\text{if}&~(\alpha_1,\alpha_2,\alpha_3,\alpha_4)&=&(2,2,3,3),\\
&16\quad &\text{if}&~(\alpha_1,\alpha_2,\alpha_3,\alpha_4)&=&(2,2,3,6).\\ 
\end{array}}
\end{cases}
\end{equation}
We will prove that remaining 11 quaternary sums of triangular numbers represent all nonnegative integers except 1 (see Table 1).

Now, we classify quinary almost universal sums of triangular numbers with one exception 1.
From \eqref{truant}, we have $\alpha_4\leq\alpha_5\leq\mathfrak{T}_2(\Delta_{\alpha_1,\alpha_2,\alpha_3,\alpha_4})$ for each possible case.
So there are 36 candidates of quinary almost universal sums of triangular numbers with one exception 1.
One may easily check that there are 12 sums among them that are almost universal with one exception 1 but not proper.
Furthermore, if $\alpha_5=\mathfrak{T}(\Delta_{\alpha_1,\alpha_2,\alpha_3,\alpha_4})-1$ for each possible case, then the sum $\Delta_{\alpha_1,\alpha_2,\alpha_3,\alpha_4,\alpha_5}$ is not almost universal with one exception since it doesn't represent $1$ and $\mathfrak{T}_2(\Delta_{\alpha_1,\alpha_2,\alpha_3,\alpha_4})$.
We will prove that remaining 18 quinary sums of triangular numbers represent all nonnegative integers except 1 (see Table 1).

Finally, for $k\geq6$ we classify all $k$-ary almost universal sums of triangular numbers with one exception 1.
For each possible case, since $\mathfrak{T}_2(\Delta_{\alpha_1,\alpha_2,\dots\alpha_{k-1}})=\alpha_{k-1}+1$, 
we know that $\alpha_{k-1}\leq\alpha_k\leq\alpha_{k-1}+1$. 
Note that $\Delta_{\alpha_1,\alpha_2,\dots,\alpha_{k-1},\alpha_{k-1}}$ is not almost universal with one exception since it doesn't represent $1$ and $\alpha_{k-1}+1$.
Furthermore, one may directly check that $\Delta_{\alpha_1,\alpha_2,\dots,\alpha_{k-1},\alpha_{k-1}+1}$ is almost universal with one exception 1 but not proper.
Therefore, there are no $k$-ary proper almost universal sums of triangular numbers with one exception $1$ for any integer $k\geq6$ (see Table 1).

Now, we prove that above 11 quaternary and 18 quinary sums of triangular numbers are proper almost universal with one exception $1$.
In all cases, it is enough to show that each sum represents all nonnegative integers except 1 since its properness is clear.
\begin{table}[hpt!]
\caption{Proper almost universal sums with one exception 1}
\vspace{-3mm}\begin{center}
\renewcommand{\arraystretch}{0.9}\renewcommand{\tabcolsep}{1mm}
\begin{tabular}{l|l|l}
\hline
Sums &Candidates & Conditions on $\alpha_k$ \\
\hline
$\Delta_{2,2,2,\alpha_4}$ & $2\leq \alpha_4 \leq 3$ & $\alpha_4\neq2$\\
$\Delta_{2,2,3,\alpha_4}$ & $3\leq \alpha_4 \leq 10$ & $\alpha_4\neq3,6,9$\\
$\Delta_{2,3,3,\alpha_4}$ & $3\leq \alpha_4 \leq 4$ & $\alpha_4\neq3$\\
$\Delta_{2,3,4,\alpha_4}$ & $4\leq \alpha_4 \leq 8$ & $\alpha_4\neq7$\\
\hline
$\Delta_{2,2,2,2,\alpha_5}$ & $2\leq \alpha_5 \leq 3$ & $\alpha_5\neq2,3^{\dagger}$\\
$\Delta_{2,2,3,3,\alpha_5}$ & $3\leq \alpha_5 \leq 19$ & $\alpha_5\neq4^{\dagger},5^{\dagger},7^{\dagger},8^{\dagger},10^{\dagger},18$\\
$\Delta_{2,2,3,6,\alpha_5}$ & $6\leq \alpha_5 \leq 16$ & $\alpha_5\neq7^{\dagger},8^{\dagger},10^{\dagger},15$\\
$\Delta_{2,2,3,9,\alpha_5}$ & $9\leq \alpha_5 \leq 10$ & $\alpha_5\neq9,10^{\dagger}$\\
$\Delta_{2,3,3,3,\alpha_5}$ & $3\leq \alpha_5 \leq 4$ & $\alpha_5\neq3,4^{\dagger}$\\
$\Delta_{2,3,4,7,\alpha_5}$ & $7\leq \alpha_5 \leq 8$ & $\alpha_5\neq7,8^{\dagger}$\\
\hline
$\Delta_{\alpha_1,\alpha_2,\dots,\alpha_{k-1},\alpha_k}(k\geq6)$ & $\alpha_{k-1}\leq\alpha_k\leq\alpha_{k-1}+1$ & $\alpha_k\neq\alpha_{k-1},\alpha_{k-1}+1^{\dagger}$\\
\hline
\end{tabular}
\end{center}
\end{table}

\vspace{-0.5pc}
\noindent\textbf{(i)} Let $(\alpha_1,\alpha_2,\alpha_3)=(2,2,2)$. 
We show that $\Delta_{2,2,2,3}$ is an almost universal sum of triangular numbers with one exception 1.
Since $\Delta_{1,1,1}$ is universal, $\Delta_{2,2,2}$ represents all nonnegative even integers.
Let $n$ be an odd integer greater than 1.
Since $n-3$ is represented by $\Delta_{2,2,2}$, $n$ is represented by $\Delta_{2,2,2,3}$.
Therefore, $\Delta_{2,2,2,3}$ is an almost universal sum of triangular numbers with one exception 1.
\vskip0.5pc
\noindent\textbf{(ii)} Let $(\alpha_1,\alpha_2,\alpha_3)=(2,2,3)$. We show that $\Delta_{2,2,3,\alpha_4}$ ($3\leq \alpha_4\leq 10$, $\alpha_4\neq3,6,9$) are almost universal sums of triangular numbers with one exception $1$. 
Since the proofs are quite similar to each other, we only provide the proof of $\Delta_{2,2,3,4}$.
By Equation \eqref{congruence condition}, it suffices to show that the equation
\begin{equation}\label{2234}
2x^2+2y^2+3z^2+4t^2 = 8n+11
\end{equation}
has an integer solution $(x,y,z,t)\in\mathbb{Z}^4$ such that  $xyzt\equiv1 \Mod2$ for any nonnegative integer $n$ except 1.
If $n=0$ or $2\leq n\leq3$, then one may directly check that Equation \eqref{2234} has a desired integer solution.
Therefore, we may assume that $n\geq4$. 
Note that the genus of $f(x,y,z)=2(4x+y)^2+2y^2+3z^2$ consists of 
$$
M_f=\langle3,4,16\rangle\quad\text{and}\quad M_2=\langle4\rangle\perp\begin{pmatrix}7&1\\1&7\end{pmatrix}.
$$
For a nonnegative integer $m$, if $m\equiv 7\Mod8$ and $m\neq3^{2u+1}(3v+2)$ for any nonnegative integers $u$ and $v$ , then $m$ is represented by $M_f$ or $M_2$ by 102:5 of \cite{om}, for it is represented by $M_f$ over $\mathbb{Z}_p$ for any prime $p$.
One may easily check that
$$
M_2\prec_{8,7}M_f.
$$
Note that $8n+11-4d^2\equiv7\Mod8$ and $8n+11-4d^2\not\equiv0\Mod3$, where
$$
d=\begin{cases}
1\quad\text{if}~ 8n+11\equiv0 \Mod 3,\\
3 \quad\text{if}~ 8n+11\not\equiv0 \Mod 3.\\
\end{cases}
$$
Furthermore, since we are assuming $n\geq4$, $8n+11-4d^2$ is positive.
Therefore, the equation
$$
2x^2+2y^2+3z^2=8n+11-4d^2
$$
has an integer solution $(x,y,z)\in\mathbb{Z}^3$ such that $x\equiv y\Mod4$ by Theorem \ref{good}.
This completes the proof.

\vskip0.5pc
\noindent\textbf{(iii)} Let $(\alpha_1,\alpha_2,\alpha_3)=(2,3,3)$.
We show that $\Delta_{2,3,3,4}$ is an almost universal sum of triangular numbers with one exception 1.
By Equation \eqref{congruence condition}, it suffices to show that the equation
\begin{equation}\label{2334}
2x^2+3y^2+3z^2+4t^2=8n+12
\end{equation}
has an integer solution $(x,y,z,t)\in\mathbb{Z}^4$ such that $xyzt\equiv1\Mod2$ for any nonnegative integer $n$ except 1.
The class number of $f(x,y,z)=2x^2+3y^2+3z^2$ is one. 
For a nonnegative integer $m$, if $m\equiv0\Mod8$ and $m\not\equiv3^{2u}(3v+1)$ for any nonnegative integers $u$ and $v$, then $m$ is represented by $f$ over $\mathbb{Z}_p$ for any prime $p$, in particular, it is primitively represented by $f$ over $\mathbb{Z}_2$. 
Let $8n+12=3^{2\ell}(8k+4)$ for some nonnegative integers $\ell$ and $k$ such that $8k+4\not\equiv0\Mod{9}$.
If $k=0$, note that for any $\ell\geq1$,
$$
2(3^{\ell-1})^2+3(3^{\ell-1})^2+3(3^{\ell})^2+4(3^{\ell-1})^2=4\cdot3^{2\ell}.
$$
If $k=2$, note that for any $\ell\geq1$,
$$
2(7\cdot3^{\ell-1})^2+3(5\cdot3^{\ell-1})^2+3(3^{\ell-1})^2+4(3^{\ell-1})^2=20\cdot3^{2\ell}.
$$
One may directly check that if $0\leq k\leq12$, $k\neq0,2$, then the equation
$$
2x^2+3y^2+3z^2+4t^2=8k+4
$$
has an integer solution $(x,y,z,t)\in\mathbb{Z}^4$ such that $xyzt\equiv1\Mod2$.
Therefore we may assume that $k\geq13$.
One may easily check that $8k+4-4d^2$ is represented by $f$ over $\mathbb{Z}_p$ for any $p$, where
$$
d=\begin{cases}
3\quad\text{if}~8k+4\equiv r\Mod9~\text{for any}~r\in\{2,3,5,6,8\},\\
1\quad\text{if}~8k+4\equiv r\Mod9~\text{for any}~r\in\{1,7\},\\
5\quad\text{if}~8k+4\equiv4\Mod9,
\end{cases}
$$
in particular, it is primitively represented by $f$ over $\mathbb{Z}_2$. 
Furthermore, since we are assuming $k\geq12$, $8k+4-4d^2$ is positive.
By 102:5 of \cite{om}, the equation 
$$
2x^2+3y^2+3z^2=8k+4-4d^2
$$ 
has an integer solution $(x,y,z)\in\mathbb{Z}^3$ such that $xyz\equiv1\Mod2$. 
This completes the proof.
\vskip0.5pc
\noindent\textbf{(iv)} Let $(\alpha_1,\alpha_2,\alpha_3)=(2,3,4)$.
We show that $\Delta_{2,3,4,\alpha_4}$ ($4\leq\alpha_4\leq8$, $\alpha_4\neq7$) are almost universal sums of triangular numbers with one exception $1$.

Assume $\alpha_4=4,6$.
Since the proofs are quite similar to each other, we only provide the proof of $\Delta_{2,3,4,4}$.
By Equation \ref{congruence condition}, it suffices to show that the equation
\begin{equation}\label{2344}
2x^2+3y^2+4z^2+4t^2=8n+13
\end{equation}
has an integer solution $(x,y,z,t)\in\mathbb{Z}^4$ such that $xyzt\equiv1\Mod2$ for any nonnegative integer $n$ except 1. 
If $n=0$ or $2\leq n\leq 11$, then one may directly check that Equation \eqref{2344} has a desired integer solution.
Therefore, we may assume that $n\geq12$.
Note that the genus of $f(x,y,z)=2x^2+3y^2+4z^2$ consists of 
$$
f(x,y,z)\quad \text{and}\quad g(x,y,z)=x^2+2y^2+12z^2.
$$
For a nonnegative integer $m$, if $m\equiv1\Mod8$, then $m$ is represented by $M_f$ or $M_2$ by 102:5 of \cite{om}, for it is represented by $M_f$ over $\mathbb{Z}_p$ for any prime $p$. 
One may easily show that
$$
B_f(g,5,1)=\{\pm(1,0,0)\}\quad\text{and}\quad B_f(g,5,4)=\{\pm(2,0,0)\}.
$$
In each case, if we define $\displaystyle T=\begin{pmatrix}5&0&0\\0&1&-12\\0&2&1\end{pmatrix}$, then one may easily show that it satisfies all conditions in Theorem \ref{pme}.
Note that $z=\pm(1,0,0)$ are the only integral primitive eigenvectors of $T$. 
Therefore, we have
$$
S_{5,r}\cap Q(g)\setminus\{s^2|s\in\mathbb{Z}\}\subset Q(f), 
$$
for any $r\in\{1,4\}$.
Since $g$ is contained in the spinor genus of $f$, every square $t^2$ of an integer that has a prime divisor greater than 3 is represented by $f$ by Lemma 2.4 in \cite{ternary}.
If $t$ is divisible by $2$(3), then $t^2$ is represented by $f$ since $4(9, \text{ respectively})$ is represented by $f$. 
Therefore, every integer greater than 1 that is congruent to $1$ modulo $8$ and congruent to $1$ or $4$ modulo $5$ is represented by $f$.
Note that $8n+13-4d^2$ is congruent to $1$ modulo $8$ and congruent to $1$ or $4$ modulo $5$, where 
$$
d=\begin{cases}
1\quad\text{if}~8n+13\equiv r\Mod5,~\text{for any}~r\in\{0,3\},\\
5\quad\text{if}~8n+13\equiv r\Mod5,~\text{for any}~r\in\{1,4\},\\
3\quad\text{if}~8n+13\equiv2\Mod5.
\end{cases}
$$
Furthermore, since we are assuming $n\geq12$, $8n+13-4d^2\geq2$.
Therefore, the equation
$$
2x^2+3y^2+4z^2=8n+13-4d^2
$$
has an integer solution.
This completes the proof. 

Assume $\alpha_4=5$.
By Equation \eqref{congruence condition}, It suffices to show that the equation 
\begin{equation}\label{2345}
2x^2+3y^2+4z^2+5t^2=8n+14
\end{equation}
has an integer solution $(x,y,z,t)\in\mathbb{Z}^4$ such that $xyzt\equiv1\Mod2$ for any nonnegative integer $n$ except 1. 
If $n=0$ or $2\leq n\leq 218$, then one may directly check that Equation \eqref{2345} has a desired integer solution. 
Therefore we may assume that $n\geq 219$. 
Note that the genus of $f(x,y,t)=2(2x+y)^2+3y^2+5t^2$ consists of 
$$
M_f=\begin{pmatrix}5&1\\1&5\end{pmatrix}\perp\langle5\rangle\quad\text{and}\quad M_2=\langle1,1,120\rangle.
$$
For a nonnegative integer $m$, if $m\equiv2\Mod8$ and $m\not\equiv0\Mod3$, then $m$ is represented by $M_f$ or $M_2$ by 102:5 of \cite{om}, for it is represented by $M_f$ over $\mathbb{Z}_p$ for any prime $p$. 
One may easily check that  
$$
M_2\prec_{7,r}M_f
$$
for any $r\in\{0,3,5,6\}$.
Assume that $8n+14\not\equiv0\Mod3$. 
Note that $8n+14-4d^2\equiv2 \Mod8$, $8n+14-4d^2\not\equiv0\Mod3$, and $8n+14-4d^2\equiv r\Mod7, ~\text{for some}~ r\in\{0,3,5,6\}$, where
$$
d=\begin{cases}
\aligned
&21&\quad&\text{if}~8n+14\equiv r\Mod7 ~\text{for any}~ r\in\{0,3,5,6\},\\
&3&\quad&\text{if}~8n+14\equiv r\Mod7~\text{for any}~r\in\{1,4\},\\
&9&\quad&\text{if}~8n+14\equiv 2\Mod7.
\endaligned
\end{cases}
$$
Assume that $8n+14\equiv0\Mod3$.
Note that $8n+14-4d^2\equiv2\Mod8$, $8n+14-4d^2\not\equiv0\Mod3$, and $8n+14-4d^2\equiv r\Mod7, ~\text{for some}~ r\in\{0,3,5,6\}$, where
$$
d=\begin{cases}
\aligned
&7&\quad&\text{if}~8n+14\equiv r\Mod7 ~\text{for any}~ r\in\{0,3,5,6\},\\
&1&\quad&\text{if}~8n+14\equiv r\Mod7~\text{for any}~ r\in\{2,4\},\\
&5&\quad&\text{if}~8n+14\equiv 1\Mod7.
\endaligned
\end{cases}
$$
Furthermore, since we are assuming $n\geq219$, $8n+14-4d^2$ is positive.
Therefore, the equation
$$
2x^2+3y^2+5t^2=8n+14-4d^2
$$
has an integer solution by Theorem \ref{good}. 
This completes the proof.

Assume $\alpha_4=8$.
By Equation \eqref{congruence condition}, it suffices to show that the equation
\begin{equation}\label{2348}
2x^2+3y^2+4z^2+8t^2=8n+17
\end{equation}
has an integer solution $(x,y,z,t)\in\mathbb{Z}^4$ such that $xyzt\equiv1\Mod2$ for any nonnegative integer $n$ except 1.
If $n=0$, then one may directly check that Equation \eqref{2348} has a desired integer solution. 
Therefore, we may assume that $n\geq2$. 
Note that the class number of $f(x,z,t)=2x^2+4(2z+t)^2+8t^2$ is one.
For a nonnegative integer $m$, if $m\equiv14\Mod{16}$, then it is represented by $f$.
Note that $8n+17-3d^2\equiv14\Mod{16}$, where
$$
d=\begin{cases}
1\quad\text{if}~n\equiv0\Mod2,\\
3\quad\text{if}~n\equiv1\Mod2.
\end{cases}
$$
Furthermore, since we are assuming $n\geq2$, $8n+17-3d^2$ is positive.
Therefore, the equation
$$
2x^2+4z^2+8t^2=8n+17-3d^2
$$
has an integer solution such that $z\equiv t\Mod2$. 
This completes the proof.

\vskip0.5pc
\noindent\textbf{(v)} Let $(\alpha_1,\alpha_2,\alpha_3,\alpha_4)=(2,2,3,3)$.
We show that $\Delta_{2,2,3,3,\alpha_5}$ ($3\leq\alpha_5\leq19$, $\alpha_5\neq4,5,7,8,10,18$) are almost universal sums of triangular numbers with one exception 1.

If $\alpha_5\not\equiv0\Mod3$, then the proofs are quite similar to the proof of $\Delta_{2,2,3,4}$ in the case (ii).

Assume that $\alpha_5\equiv0\Mod3$.
Since the proofs are quite similar to each other. we only provide the proof of $\Delta_{2,2,3,3,3}$.
By Equation \eqref{congruence condition}, it suffices to show that
\begin{equation}\label{22333}
2x^2+2y^2+3z^2+3t^2+3s^2=8n+13
\end{equation}
has an integer solution $(x,y,z,t,s)\in\mathbb{Z}^5$ such that $xyzts\equiv1\Mod2$ for any nonnegative integer $n$ except 1.
Let $8n+13=3^{2\ell}(8k+5)$ for some nonnegative integers $\ell$ and $k$ such that $8k+5\not\equiv0 \Mod{9}$. 
For the case when $k=0$, note that for any $\ell\geq1$, 
$$
2(3^{\ell})^2+2(3^{\ell})^2+3(3^{\ell-1})^2+3(3^{\ell-1})^2+3(3^{\ell-1})^2=5\cdot3^{2\ell}.
$$
For the case when $k=2$, note that for any $\ell\geq1$,
$$
2(3^{\ell+1})^2+2(3^{\ell})^2+3(3^{\ell-1})^2+3(3^{\ell-1})^2+3(3^{\ell-1})^2=21\cdot3^{2\ell}.
$$
If $0\leq k\leq6$ and $k\neq0,2$, then one may directly check that the equation
$$
2x^2+2y^2+3z^2+3t^2+3s^2=8k+5
$$
has an integer solution $(x,y,z,t,s)\in\mathbb{Z}^5$ such that $xyzts\equiv1\Mod2$.
Therefore, we may assume that $k\geq7$.  
Note that $8k+5-3d^2-3e^2\equiv7\Mod8$ and $8k+5-3d^2-3e^2\neq3^{2u+1}(3v+2)$ for any nonnegative integers $u$ and $v$, where
$$
(d,e)=\begin{cases}
\aligned
&(1,1)\quad &\text{if}&~ 8k+5\not\equiv0\Mod3,\\
&(3,3)\quad &\text{if}&~ 8k+5\equiv3\Mod9,\\
&(1,3)\quad &\text{if}&~ 8k+5\equiv6\Mod9.
\endaligned
\end{cases}
$$
Furthermore, since we are assuming that $k\geq7$, $8k+5-3d^2-3e^2$ is positive and it is represented by $2(4x+y)^2+2y^2+3z^2$ (see the proof of $\Delta_{2,2,3,4}$ in the case (ii)). 
Therefore, the equation
$$
2x^2+2y^2+3z^2=8k+5-3d^2-3e^2
$$
has an integer solution such that $x\equiv y \Mod4$. 
This completes the proof.

\vskip0.5pc
\noindent\textbf{(vi)} Let $(\alpha_1,\alpha_2,\alpha_3,\alpha_4)=(2,2,3,6)$.
We show that $\Delta_{2,2,3,6,\alpha_5}$ ($6\leq\alpha_5\leq16,~\alpha_5\neq7,8,10,15$) are almost universal sums of triangular numbers with one exception $1$.

If $\alpha_5\not\equiv0\Mod3$, then the proofs are quite similar to the proof of  $\Delta_{2,2,3,4}$ in the case (ii).

If $\alpha_5=6,12$, then the proofs are quite similar to the proof of $\Delta_{2,2,3,3,3}$ in the case (v).

Assume $\alpha_5=9$.
By Equation \eqref{congruence condition}, it suffices to show that 
\begin{equation}\label{22369}
2x^2+2y^2+3z^2+6t^2+9s^2=8n+22
\end{equation}
has an integer solution $(x,y,z,t,s)\in\mathbb{Z}^5$ such that $xyzts\equiv1\Mod2$ for any nonnegative integer $n$ except 1.
If $n=0$, then one may directly check that Equation \eqref{22369} has a desired integer solution.
Therefore, we may assume that $n\geq2$.
Note that the class number of $f(z,t,s)=3(2z+t)^2+6t^2+9s^2$ is one. 
For an integer $m$, if $m\equiv2\Mod8$ and $m\equiv0\Mod3$, then $m$ is represented by $f$.
Note that $8n+22-2d^2-2e^2\equiv2\Mod8$ and $8n+22-2d^2-2e^2\equiv0\Mod3$, where
$$
(d,e)=\begin{cases}
(3,3)\quad\text{if}~8n+22\equiv0\Mod3,\\
(1,1)\quad\text{if}~8n+22\equiv1\Mod3,\\
(1,3)\quad\text{if}~8n+22\equiv2\Mod3.
\end{cases}
$$
Furthermore, since we are assuming $n\geq2$, $8n+22-2d^2-2e^2$ is positive. 
Therefore, the equation
$$
3z^2+6t^2+9s^2=8n+22-2d^2-2e^2
$$
has an integer solution $(x,y,z,t,s)\in\mathbb{Z}^5$ such that $z\equiv t\Mod2$.
This completes the proof.


\vskip0.5pc
Now, we give a proof of the first statement of Theorem \ref{1}.  
For positive integers $\alpha_1, \alpha_2, \dots , \alpha_k$, assume that a sum $\Delta_{\alpha_1,\alpha_2,\dots,\alpha_k}$ of triangular numbers represents the integers
$$
2,~3,~4,~8,~10,~16,~\text{and}~19
$$
and doesn't represent $1$.
By using the same escalation method, without loss of generality, we may assume that $\Delta_{\alpha_1,\alpha_2,\alpha_3,\alpha_4}$ is contained in the above 11 quaternary sums of proper almost universal sums of triangular numbers with one exception 1, or $\Delta_{\alpha_1,\alpha_2,\alpha_3,\alpha_4,\alpha_5}$ is contained in the above 18 quinary ones.
Furthermore,  $\alpha_1, \alpha_2, \dots , \alpha_k$ are integers greater than one since $\Delta_{\alpha_1,\alpha_2,\dots,\alpha_k}$ doesn't represent 1.
Therefore, $\Delta_{\alpha_1,\alpha_2,\dots,\alpha_k}$ is an almost universal sum of triangular numbers with one exception $1$.
This completes the proof of Theorem \ref{1}. 
\end{proof}

The proofs of Theorems \ref{2}, \ref{3}, \ref{4}, and \ref{5} are quite similar to the proof of Theorem \ref{1}.
In each proof, by using an escalation method, we find all candidates of proper almost universal sums of triangular numbers with one exception $4, 5, 8$, and $2$, respectively. 
Furthermore, we show that each candidate $\Delta_{\alpha_1,\alpha_2,\dots,\alpha_k}$ $(k=4, 5)$ is an almost universal sums of triangular numbers with one exception.
To show this, we take a suitable ternary quadratic form $f(x_1,x_2,x_3)$ related with ternary section $\Delta_{\alpha_{i_1},\alpha_{i_2},\alpha_{i_3}}$ of $\Delta_{\alpha_1,\alpha_2,\dots,\alpha_k}$  like as Equation \eqref{without congruence condition}.
After that  for sufficiently large integer $n$, we find integers  $a_4,\dots,a_k\in\mathbb{Z}$ satisfying condition \eqref{conditions}.
Finally, we directly check that $\Delta_{\alpha_1,\alpha_2,\dots,\alpha_k}$ represents all remaining small integers except a single one.

Note that in most cases, the class number of $f$ is less than or equal to 2. 
The methods for computations for representations of $f$ are categorized into the following three cases:
\begin{enumerate}
\item If  $h(f)=1$, then one may easily compute the representations of $f$ by the local-global principle similarly to the proof of $\Delta_{2,3,4,8}$ in the case (iv) of Theorem \ref{1} (see also the proof of $\Delta_{2,3,3,4}$ in the case (iii)).
\vskip0.5pc
\item If $h(f)=2$ and  $|B|=|B_f(g,d,a)|=0$ for some integers $d$ and $a$, where $g$ is the genus mate of $f$, then one may compute the representations of $f$ by Theorem \ref{good} similarly to the proof of $\Delta_{2,2,3,4}$ in the case (ii) of Theorem \ref{1} (see also the proof of $\Delta_{2,3,4,5}$ in the case (iv)).
\vskip0.5pc
\item If $h(f)=2$ and $|B|\neq 0$, then then one may compute the representations of $f$ by Theorem \ref{pme} similarly to the proof of $\Delta_{2,3,4,4}$ in the case (iv) of Theorem \ref{1}.
\end{enumerate}   

In the remaining sections, since most of proofs require laborious computation, we only provide all parameters for the computations for the representations of the ternary quadratic form $f$ (see Section 4, 5, 6, and 7).
One may easily apply the given parameters to Theorem \ref{good} or Theorem \ref{pme} to compute the representations of the ternary quadratic form $f$.
We left to the readers to find suitable integers $\alpha_4,\cdots,\alpha_k$ stated above and to check the representations of remaining small integers. 
For the complete list of proper almost universal sums of triangular numbers with one exception, see Tables 1, 2, 10, 13, and 15.


\section{Proof of Theorem \ref{2}}
We give a proof of Theorem \ref{2}.
From a similar escalation method as the proof of Theorem \ref{1}, we find all candidates of 127 quaternary and 11 quinary proper almost universal sums of   triangular numbers with one exception $4$ (see Table 2).
\begin{table}[hpt!]
\caption{Proper almost universal sums with one exception 4}
\vspace{-3mm}\begin{center}
\renewcommand{\arraystretch}{0.85}\renewcommand{\tabcolsep}{0.5mm}
\begin{tabular}{l|l|l}
\hline
Sums &Candidates & Conditions on $\alpha_k$ \\
\hline
$\Delta_{1,2,5,\alpha_4}$ & $5\leq \alpha_4 \leq 19$ & $\alpha_4\neq10,15$\\
$\Delta_{1,2,6,\alpha_4}$ & $6\leq \alpha_4 \leq 50$ & $\alpha_4\neq46$\\
$\Delta_{1,2,7,\alpha_4}$ & $7\leq \alpha_4 \leq 11$ & $\alpha_4\neq7$\\
$\Delta_{1,2,8,\alpha_4}$ & $8\leq \alpha_4 \leq 19$ & $\alpha_4\neq15$\\
$\Delta_{1,2,9,\alpha_4}$ & $9\leq \alpha_4 \leq 46$ & $\alpha_4\neq42$\\
$\Delta_{1,2,10,\alpha_4}$ & $10\leq \alpha_4 \leq 14$ & $\alpha_4\neq10$\\
$\Delta_{1,2,11,\alpha_4}$ & $11\leq \alpha_4 \leq 25$ & $\alpha_4\neq21$\\
\hline
$\Delta_{1,2,5,10,\alpha_5}$ & $10\leq \alpha_5 \leq 29$ & $\alpha_5\neq11^{\dagger},12^{\dagger},13^{\dagger},14^{\dagger},16^{\dagger},$\\
~&~& \hspace{8.5mm}$17^{\dagger},18^{\dagger},19^{\dagger},25$\\
$\Delta_{1,2,5,15,\alpha_5}$ & $15\leq \alpha_5 \leq 19$ & $\alpha_5\neq15,16^{\dagger},17^{\dagger},18^{\dagger},19^{\dagger}$\\
$\Delta_{1,2,6,46,\alpha_5}$ & $46\leq \alpha_5 \leq 50$ & $\alpha_5\neq46,47^{\dagger},48^{\dagger},49^{\dagger},50^{\dagger}$\\
$\Delta_{1,2,7,7,\alpha_5}$ & $7\leq \alpha_5 \leq 11$ & $\alpha_5\neq7,8^{\dagger},9^{\dagger},10^{\dagger},11^{\dagger}$\\
$\Delta_{1,2,8,15,\alpha_5}$ & $15\leq \alpha_5 \leq 19$ & $\alpha_5\neq15,16^{\dagger},17^{\dagger},18^{\dagger},19^{\dagger}$\\
$\Delta_{1,2,9,42,\alpha_5}$ & $42\leq \alpha_5 \leq 46$ & $\alpha_5\neq42,43^{\dagger},44^{\dagger},45^{\dagger},46^{\dagger}$\\
$\Delta_{1,2,10,10,\alpha_5}$ & $10\leq \alpha_5 \leq 14$ & $\alpha_5\neq10,11^{\dagger},12^{\dagger},13^{\dagger},14^{\dagger}$\\
$\Delta_{1,2,11,21,\alpha_5}$ & $21\leq \alpha_5 \leq 25$ & $\alpha_5\neq21,22^{\dagger},23^{\dagger},24^{\dagger},25^{\dagger}$\\
\hline
$\Delta_{\alpha_1,\alpha_2,\dots,\alpha_{k-1},\alpha_k}(k\geq6)$ & $\alpha_{k-1}\leq\alpha_k\leq\alpha_{k-1}+4$ & $\alpha_k\neq\alpha_{k-1}+\ell$  $(\ell=0,1^{\dagger},2^{\dagger},3^{\dagger},4^{\dagger})$\\                                                                                                            
\hline
\end{tabular}
\end{center}
\end{table}

\vskip-0.5pc
Since the proof of almost universality of each candidate is quite similar to the proof of Theorem \ref{1}, we only provide all parameters for the computations for representations of the ternary quadratic form $f$ (see Tables 3, 4, 5, 6, 7, 8, and  9).
\begin{table}[h]
\caption{Data for the proof of the candidates when $(\alpha_1,\alpha_2,\alpha_3)=(1,2,5)$}
\vspace{-3mm}\begin{center}
\footnotesize
\renewcommand{\arraystretch}{1}\renewcommand{\tabcolsep}{1mm}
\begin{tabular}{|c|c|l|}
\hline
\multicolumn{2}{|c|}{$\Delta_{\alpha_{i_1},\alpha_{i_2},\alpha_{i_3}}$} & \multirow{2}{*}{Sufficient conditions for $m\rightarrow f$}\\ \cline{1-2}
$f$ & $h(f)$ &                                                                                                                     \\
\hline\hline
\multicolumn{2}{|c|}{$\Delta_{1,2,5}$}&\multirow{2}{*}{$m\equiv0\Mod8$, ~$m\neq 5^{2u+1}(5v\pm2)$}\\ \cline{1-2}
$x^2+2y^2+5z^2$ & $1$ &                                                                      \\
\hline
\end{tabular}
\end{center}
\end{table}

\begin{table}[h]
\vspace{-1mm}
\caption{Data for the proof of the candidates when $(\alpha_1,\alpha_2,\alpha_3)=(1,2,6)$}
\vspace{-4mm}\begin{center}
\footnotesize
\renewcommand{\arraystretch}{1.1}\renewcommand{\tabcolsep}{0.5mm}
\begin{tabular}{|c|c|c|c|c|c|c|l|}
\hline
$\alpha_4$   & $\Delta_{\alpha_{i_1},\alpha_{i_2},\alpha_{i_3}}$   &  \multirow{3}{*}{$d$}  & \multirow{3}{*}{$a$} & \multirow{3}{*}{$|B|$}      & \multicolumn{2}{c|}{\multirow{2}{*}{$T$}}  & \multirow{3}{*}{\setlength\arraycolsep{0pt}$\begin{array}{l} \text{Sufficient conditions} \\ \text{for } m\rightarrow f\end{array}$}\\ \cline{1-2}
$f$          &  $h(f)$                                             &                         &                      &                            &  \multicolumn{2}{c|}{}                     &                                                                                                          \\ \cline{1-2} \cline{6-7}
$M_f$        &  $M_2$                                              &                         &                      &                            &  $z$            &  $Q(z)$                  &                                                                                                          \\ 
\hline\hline

$\alpha_4\not\equiv0\Mod5$                                                  & $\Delta_{1,2,6}$                         & \multirow{4}{*}{$20$} & \multirow{4}{*}{$\begin{array}{c} 5,13,17\end{array}$} &\multirow{4}{*}{0}                        & \multicolumn{2}{|c|}{}   & \multirow{4}{*}{$\setlength\arraycolsep{0pt}\begin{array}{l} m\equiv1\Mod8,\\  m\equiv0,2,3\Mod5\end{array}$}\\ \cline{1-2}
$(2x+y)^2+2y^2+6z^2$                                                        &     $2$                                  &  &   &   & \multicolumn{2}{|c|}{}  &    \\ \cline{1-2}
\multirow{2}{*}{$\begin{pmatrix}3&1\\1&3\end{pmatrix}\perp\langle6\rangle$} & \multirow{2}{*}{$\langle1,2,24\rangle$}  &  &   &                          & \multicolumn{2}{|c|}{}  &    \\
\hline

\end{tabular}
\end{center}
\end{table}


\begin{table}[h]
\vspace{-2.5mm}
\caption{Data for the proof of the candidates when $(\alpha_1,\alpha_2,\alpha_3)=(1,2,7)$}
\vspace{-4mm}\begin{center}
\footnotesize
\renewcommand{\arraystretch}{1.01}\renewcommand{\tabcolsep}{0.5mm}
\begin{tabular}{|c|c|c|c|c|l|}
\hline
$\alpha_4$   & $\Delta_{\alpha_{i_1},\alpha_{i_2},\alpha_{i_3}}$   &  \multirow{3}{*}{$d$} & \multirow{3}{*}{$a$} & \multirow{3}{*}{$|B|$}        &  \multirow{3}{*}{\setlength\arraycolsep{0pt}$\begin{array}{l} \text{Sufficient conditions} \\ \text{for } m\rightarrow f\end{array}$}\\ \cline{1-2}
$f$          &  $h(f)$                                             &  & &      &                                                                                                           \\ \cline{1-2} 
$M_f$        &  $M_2$                                              &  & &     &                                                                                                            \\
\hline\hline

$\alpha_4=10$                 & $\Delta_{1,2,10}$   &  & &    &  \multirow{3}{*}{\setlength\arraycolsep{0pt}$\begin{array}{l}m\equiv5\Mod8,\\ m\not\equiv0\Mod5\end{array}$}\\ \cline{1-2}
$x^2+2y^2+10t^2$              & 1                  &    & &   &                                                                                                           \\ \cline{1-2} 
$\langle1,2,10\rangle$        &                   &    & &  &                                                                                                            \\
\hline\hline

$\alpha_4=8,9,11$                                                           & $\Delta_{1,2,7}$                                                       & \multirow{5}{*}{80} &  \multirow{5}{*}{$\begin{array}{c}10,26,34,\\50,66,74\end{array}$}  &\multirow{5}{*}{0}                                                   &  \multirow{5}{*}{$\setlength\arraycolsep{0pt}\begin{array}{l}m\equiv2\Mod8,\\ m\equiv0,1,4\Mod5,\\ m\not\equiv 0\Mod7\end{array}$}\\
\hline

\end{tabular}
\end{center}
\end{table}

\begin{table}[h]
\vspace{-2.5mm}
\caption{Data for the proof of the candidates when  $(\alpha_1,\alpha_2,\alpha_3)=(1,2,8)$}
\vspace{-4mm}
\begin{center}
\footnotesize
\renewcommand{\arraystretch}{1.01}\renewcommand{\tabcolsep}{0.5mm}
\begin{tabular}{|c|c|c|c|c|l|}
\hline
$\alpha_4$   & $\Delta_{\alpha_{i_1},\alpha_{i_2},\alpha_{i_3}}$   &  \multirow{3}{*}{$d$} & \multirow{3}{*}{$a$} & \multirow{3}{*}{$|B|$}        &  \multirow{3}{*}{$\setlength\arraycolsep{0pt}\begin{array}{l} \text{Sufficient conditions} \\ \text{for } m\rightarrow f\end{array}$}\\ \cline{1-2}
$f$          &  $h(f)$                                             &  & &      &                                                                                                           \\ \cline{1-2} 
$M_f$        &  $M_2$                                              &  & &     &                                                                                                            \\
\hline\hline

$\alpha_4=10$                 & $\Delta_{1,2,10}$   & & &      &  \multirow{3}{*}{$\setlength\arraycolsep{0pt}\begin{array}{l}m\equiv5\Mod8,\\ m\not\equiv0\Mod5\end{array}$}\\ \cline{1-2}
$x^2+2y^2+10t^2$              & 1                  &  & &    &                                                                                                           \\ \cline{1-2} 
$\langle1,2,10\rangle$        &                   &   & &   &                                                                                                            \\
\hline\hline

$\alpha_4\not\equiv0\Mod5$                                                           & $\Delta_{1,2,8}$                                                       & \multirow{5}{*}{40} &  \multirow{5}{*}{$\begin{array}{c}11,19,35\end{array}$} & \multirow{5}{*}{0}                          &  \multirow{5}{*}{$\setlength\arraycolsep{0pt}\begin{array}{l}m\equiv3\Mod8,\\ m\equiv0,1,4\Mod5\end{array}$}\\ \cline{1-2}
$x^2+2(2y+z)^2+8z^2$                                                        & 2                                                                      &  & &    &                                                                                                           \\
\hline
\end{tabular}
\end{center}
\end{table}


\begin{table}[hpt!]
\vspace{-1.5mm}
\caption{Data for the proof of the candidates when $(\alpha_1,\alpha_2,\alpha_3)=(1,2,9)$}
\vspace{-4mm}
\begin{center}
\footnotesize
\renewcommand{\arraystretch}{0.91}\renewcommand{\tabcolsep}{0.5mm}
\begin{tabular}{|c|c|c|c|c|c|c|l|}
\hline
$\alpha_4$                     & $\Delta_{\alpha_{i_1},\alpha_{i_2},\alpha_{i_3}}$   &  \multirow{3}{*}{$d$} & \multirow{3}{*}{$a$} & \multirow{3}{*}{$|B|$}       & \multicolumn{2}{c|}{\multirow{2}{*}{$T$}}  & \multirow{3}{*}{$\setlength\arraycolsep{0pt}\begin{array}{l}\text{Sufficient conditions} \\ \text{for } m\rightarrow f\end{array}$}\\ \cline{1-2}
$f$                            &  $h(f)$                                             &                       &                       &                             &  \multicolumn{2}{c|}{}                     &                                                                                                          \\ \cline{1-2} \cline{6-7}
$M_f$                          &  $M_2$                                              &                      &                        &                            &  $z$            &  $Q(z)$                  &                                                                                                          \\
\hline\hline

$\alpha_4=21$                  & $\Delta_{1,9,21}$                                 &                       &                        &                 & \multicolumn{2}{|c|}{}   & \multirow{3}{*}{$\setlength\arraycolsep{0pt}\begin{array}{l}m\equiv7\Mod8,\\ m=3^u(3v+1),(u\leq1)\end{array}$}\\ \cline{1-2}
$x^2+9z^2+21t^2$               &     $1$                                           &                       &                        &              & \multicolumn{2}{|c|}{}  &    \\ \cline{1-2} 
$\langle1,9,21\rangle$         &                                                   &                       &                        &                     &\multicolumn{2}{|c|}{}             &    \\
\hline
\end{tabular}
\end{center}
\end{table}


\begin{table}[h]
\vspace{-2mm}
\caption{Data for the proof of the candidates when $(\alpha_1,\alpha_2,\alpha_3)=(1,2,10)$}
\vspace{-3mm}
\begin{center}
\footnotesize
\renewcommand{\arraystretch}{0.94}\renewcommand{\tabcolsep}{0.5mm}
\begin{tabular}{|c|c|l|}
\hline

\multicolumn{2}{|c|}{$\Delta_{\alpha_{i_1},\alpha_{i_2},\alpha_{i_3}}$} & \multirow{2}{*}{Sufficient conditions for $m\rightarrow f$}\\ \cline{1-2}
$f$ & $h(f)$ &                                                                                                                     \\
\hline\hline
\multicolumn{2}{|c|}{$\Delta_{1,2,10}$}&\multirow{2}{*}{$m\equiv5\Mod8, m\not\equiv0\Mod5$}\\ \cline{1-2}
$x^2+2y^2+10z^2$  & $1$ &                                                                      \\
\hline

\end{tabular}
\end{center}
\end{table}


\begin{table}[h]
\vspace{-2mm}
\caption{Data for the proof of the candidates when $(\alpha_1,\alpha_2,\alpha_3)=(1,2,11)$}
\vspace{-4mm}
\begin{center}
\footnotesize
\renewcommand{\arraystretch}{0.9}\renewcommand{\tabcolsep}{0.4mm}
\begin{tabular}{|c|c|c|c|c|c|c|l|}
\hline
$\alpha_4$                     & $\Delta_{\alpha_{i_1},\alpha_{i_2},\alpha_{i_3}}$   &  \multirow{3}{*}{$d$} & \multirow{3}{*}{$a$} & \multirow{3}{*}{$|B|$}       & \multicolumn{2}{c|}{\multirow{2}{*}{$T$}}  & \multirow{3}{*}{$\setlength\arraycolsep{0pt}\begin{array}{l}\text{Sufficient conditions} \\ \text{for } m\rightarrow f\end{array}$}\\ \cline{1-2}
$f$                            &  $h(f)$                                             &                       &                       &                             &  \multicolumn{2}{c|}{}                     &                                                                                                          \\ \cline{1-2} \cline{6-7}
$M_f$                          &  $M_2$                                              &                      &                        &                            &  $z$            &  $Q(z)$                  &                                                                                                          \\
\hline\hline

$\alpha_4\not\equiv0\Mod5$                                 & $\Delta_{1,2,11}$                                  &  \multirow{5}{*}{$20$} & \multirow{5}{*}{$\setlength\arraycolsep{0pt}\begin{array}{c}2,10,\\18\end{array}$} & \multirow{5}{*}{$0$}       & \multicolumn{2}{c|}{}              & \multirow{5}{*}{$\setlength\arraycolsep{0pt}\begin{array}{l}m\equiv6\Mod8,\\ m\equiv0,2,3\Mod5\end{array}$}\\ \cline{1-2}
$(2x+y)^2+2y^2+11z^2$                                      &  2                                                &                         &                                                           &                             &  \multicolumn{2}{c|}{}                     &                                                                                                       \\
\hline
\end{tabular}
\end{center}
\end{table}


\section{Proof of Theorem \ref{3}}

In this section, we give a proof of Theorem \ref{3}.
From a similar escalation method as the proof of Theorem \ref{1}, we find all candidates of 56 quaternary and 115 quinary proper almost universal sums of triangular numbers with one exception $5$ (see Table 10).
The proof is quite similar to Theorem \ref{1} except for the cases when $(\alpha_1,\alpha_2,\alpha_3)=(1,1,8)$.

\begin{table}[htp!]
\caption{Proper almost universal sums with one exception 5}
\vspace{-3mm}\begin{center}
\renewcommand{\arraystretch}{1.0}\renewcommand{\tabcolsep}{1mm}
\begin{tabular}{l|l|l}
\hline
Sums &Candidates & Conditions on $\alpha_k$ \\
\hline
$\Delta_{1,1,6,\alpha_4}$ & $6\leq \alpha_4 \leq 14$ & $\alpha_4\neq6,9$\\
$\Delta_{1,1,7,\alpha_4}$ & $7\leq \alpha_4 \leq 26$ & $\alpha_4\neq7,14,21$\\
$\Delta_{1,1,8,\alpha_4}$ & $8\leq \alpha_4 \leq 41$ & $\alpha_4\neq30,36$\\
\hline
$\Delta_{1,1,6,6,\alpha_5}$ & $6\leq \alpha_5 \leq 59$ & $\alpha_5\neq7^{\dagger},8^{\dagger},10^{\dagger},11^{\dagger},12^{\dagger},13^{\dagger},$\\
&& \hspace{8.5mm}$14^{\dagger},54$\\
$\Delta_{1,1,6,9,\alpha_5}$ & $9\leq \alpha_5 \leq 14$ & $\alpha_5\neq9,10^{\dagger},11^{\dagger},12^{\dagger},13^{\dagger},14^{\dagger}$\\
$\Delta_{1,1,7,7,\alpha_5}$ & $7\leq \alpha_5 \leq 47$ & $\alpha_5\neq8^{\dagger},9^{\dagger},10^{\dagger},11^{\dagger},12^{\dagger},13^{\dagger},$\\
&& \hspace{8.5mm}$15^{\dagger},16^{\dagger},17^{\dagger},18^{\dagger},19^{\dagger},20^{\dagger},$\\
&& \hspace{8.5mm}$22^{\dagger},23^{\dagger},24^{\dagger},25^{\dagger},26^{\dagger},42$\\
$\Delta_{1,1,7,14,\alpha_5}$ & $14\leq \alpha_5 \leq 40$ & $\alpha_5\neq15^{\dagger},16^{\dagger},17^{\dagger},18^{\dagger},19^{\dagger},20^{\dagger},$\\
&& \hspace{8.5mm}$22^{\dagger},23^{\dagger},24^{\dagger},25^{\dagger},26^{\dagger},35$\\
$\Delta_{1,1,7,21,\alpha_5}$ & $21\leq \alpha_5 \leq 26$ & $\alpha_5\neq21,22^{\dagger},23^{\dagger},24^{\dagger},25^{\dagger},26^{\dagger}$\\
$\Delta_{1,1,8,30,\alpha_5}$ & $30\leq \alpha_5 \leq 71$ & $\alpha_5\neq31^{\dagger},32^{\dagger},33^{\dagger},34^{\dagger},35^{\dagger},37^{\dagger},$\\
&&\hspace{8.5mm}$38^{\dagger},39^{\dagger},40^{\dagger},41^{\dagger},66$\\
$\Delta_{1,1,8,36,\alpha_5}$ & $36\leq \alpha_5 \leq 41$ & $\alpha_5\neq36,37^{\dagger},38^{\dagger},39^{\dagger},40^{\dagger},41^{\dagger}$\\
\hline
$\Delta_{\alpha_1,\alpha_2,\dots,\alpha_{k-1},\alpha_k}(k\geq6)$ & $\alpha_{k-1}\leq\alpha_k\leq\alpha_{k-1}+5$ & $\alpha_k\neq\alpha_{k-1}+\ell$\\
&&\quad\quad$(\ell=0,1^{\dagger},2^{\dagger},3^{\dagger},4^{\dagger},5^{\dagger})$\\
\hline
\end{tabular}
\end{center}
\end{table}

Assume $(\alpha_1,\alpha_2,\alpha_3)=(1,1,8)$.
We show that $\Delta_{1,1,8,\alpha_4}\quad(8\leq\alpha_4\leq41, ~\alpha_4\neq30,36)$ are almost universal sums of triangular numbers with one exception 5. 
Since the proofs are quite similar to each other, we only provide the proof of $\Delta_{1,1,8,8}$.
By Equation \eqref{congruence condition}, it suffices to show that the equation
\begin{equation}\label{1188}
x^2+y^2+8z^2+8t^2=8n+18
\end{equation}
has an integer solution $(x,y,z,t)\in\mathbb{Z}^4$ such that $xyzt\equiv1\Mod2$.
If $0\leq n\leq 19$, then one may directly check that Equation \eqref{1188} has a desired integer solution.
Therefore, we may assume that $n\geq20$.
Note that the genus of $f(x,y,z)=x^2+(2y+z)^2+8z^2$ consists of 
$$
M_f=\begin{pmatrix}1&0&0\\0&4&2\\0&2&9\end{pmatrix},\quad M_2=\langle1,1,32\rangle,\quad\text{and}\quad M_3=\begin{pmatrix}2&0&1\\0&2&1\\1&1&9\end{pmatrix}.
$$ 
For an integer $m$, if $m\equiv2\Mod8$, then $m$ is represented by $M_f$, $M_2$, or $M_3$ by 102:5 of \cite{om}, for it is represented by $M_f$ over $\mathbb{Z}_p$ for any prime $p$.
Furthermore, note that the spinor genus of $f$ consists of unique class $f$, itself.
One may easily show that a positive integer $a$ is a spinor exception of the genus of $f$ only if $a=2m^2$ for some $m\in\mathbb{Z}$ (for details, see \cite{sp2}).
Assume that 
$$
8n+18-8=2m_1^2 \quad\text{and}\quad 8n+18-8\cdot3^2=2m_2^2
$$ 
for some $m_1,m_2\in\mathbb{Z}$.
Then $2m_1^2-2m_2^2=64$.
So $(m_1,m_2)\in\{(9,7),(6,2)\}$.
However this is impossible since we are assuming $n\geq20$. 
Therefore, one of the integers $8n+18-8$ or $8n+18-8\cdot3^2$ is not spinor exception of the genus of $f$, in fact, it is represented by $f$. 
This implies that the equation
$$
x^2+y^2+8z^2=8n+18-8d^2
$$
has an integer solution $(x,y,z)\in\mathbb{Z}^3$ for some $d\in\{1,3\}$ such that $y\equiv z\Mod2$.
This completes the proof.

Now, We show that $\Delta_{1,1,8,30,\alpha_5}~(30\leq\alpha_5\leq71, ~\alpha_5\neq31,32,33,34,35,37,38,39$, $40,41,66)$ 
are almost universal sums of triangular numbers with one exception 5.
Similarly as above, one may easily show that $\Delta_{1,1,8,30}$ represents all nonnegative integers except $5$ and $71$.
Therefore, every $\Delta_{1,1,8,30,\alpha_5}$ is an almost universal sum of triangular numbers with one exception $5$. 
This completes the proof.

In the remaining cases, since the proof of the almost universality of each candidate is quite similar to the proof of Theorem \ref{1},
we only provide all parameters for the computations for representations of the ternary quadratic form $f$ (see Tables 11 and 12).

\begin{table}[h]
\vspace{-1mm}
\caption{Data for the proof of the candidates when $\small{(\alpha_1,\alpha_2,\alpha_3)=(1,1,6)}$}
\vspace{-3mm}\begin{center}
\footnotesize
\renewcommand{\arraystretch}{1.1}\renewcommand{\tabcolsep}{0.5mm}
\begin{tabular}{|c|c|l|}
\hline

$\alpha_4$ & $\Delta_{\alpha_{i_1},\alpha_{i_2},\alpha_{i_3}}$ & \multirow{2}{*}{Sufficient conditions for $m\rightarrow f$}\\ \cline{1-2}
$f$ & $h(f)$ &                                                                                                                     \\
\hline\hline
every case & $\Delta_{1,1,6}$&\multirow{2}{*}{$m\equiv0\Mod8, m\neq 3^{2u+1}(3v+1)$}\\ \cline{1-2}
$x^2+y^2+6z^2$  & $1$ &  \\          
\hline\hline
$\alpha_4=12$ & $\Delta_{1,1,12}$ &\multirow{2}{*}{$m\equiv6\Mod8$, $m\neq 3^{2u+1}(3v+2)$}\\ \cline{1-2}
$x^2+y^2+12t^2$  & $1$ &      \\      
\hline
\end{tabular}
\end{center}
\end{table}


\begin{table}[h]
\vspace{-3mm}
\caption{Data for the proof of the candidates when $\small{(\alpha_1,\alpha_2,\alpha_3)=(1,1,7)}$}
\vspace{-3mm}\begin{center}
\footnotesize
\renewcommand{\arraystretch}{1.1}\renewcommand{\tabcolsep}{0.5mm}
\begin{tabular}{|c|c|c|c|c|c|c|l|}
\hline
 \multicolumn{2}{|c|}{$\Delta_{\alpha_{i_1},\alpha_{i_2},\alpha_{i_3}}$}             &  \multirow{3}{*}{$d$} & \multirow{3}{*}{$a$} & \multirow{3}{*}{$|B|$}       & \multicolumn{2}{c|}{\multirow{2}{*}{$T$}}  & \multirow{3}{*}{$\setlength\arraycolsep{0pt}\begin{array}{l}\text{Sufficient conditions} \\ \text{for } m\rightarrow f\end{array}$}\\ \cline{1-2}
$f$                            &  $h(f)$                                             &                       &                       &                             &  \multicolumn{2}{c|}{}                     &                                                                                                          \\ \cline{1-2} \cline{6-7}
$M_f$                          &  $M_2$                                              &                      &                        &                            &  $z$            &  $Q(z)$                  &                                                                                                          \\
\hline  \hline

 \multicolumn{2}{|c|}{$\Delta_{1,1,7}$}                                     & \multirow{5}{*}{$8$}   & \multirow{5}{*}{$1$}    & \multirow{5}{*}{8}      & \multicolumn{2}{c|}{\multirow{4}{*}{$\begin{pmatrix}8&0&0\\0&6&-14\\0&2&6\end{pmatrix}$}} & \multirow{5}{*}{$\setlength\arraycolsep{0pt}\begin{array}{l}m>1,\\ m\notequiv0\Mod{49},\\ m\equiv1\Mod8,\\ m\neq7^{2u+1}(7v+r)\\ \hspace{6mm}\text{for any } r\in\{3,5,6\}\end{array}$}\\  \cline{1-2}
 $(2x+y)^2+2y^2+7z^2$                     &    2                             &                       &                         &                          & \multicolumn{2}{c|}{}                                                                                                   &\\ \cline{1-2}
 \multirow{3}{*}{$\langle2,2,7\rangle$} & \multirow{3}{*}{$\langle1,2,14\rangle$}&                    &                         &                          & \multicolumn{2}{c|}{}                                                                                                   &\\ 
                                        &                                   &                        &                         &                          & \multicolumn{2}{c|}{}                                                                                                   &\\ \cline{6-7} 
                                        &                                   &                        &                         &                          & $\pm(1,0,0)$              & $1$                                                                                         &\\
\hline
\end{tabular}
\end{center}
\end{table}

\section{Proof of Theorem \ref{4}}
In this section, we give a proof of Theorem \ref{4}.
From a similar escalation method as the proof of Theorem \ref{1}, we find all candidates of  7 quaternary and 73 quinary proper almost universal sums of triangular numbers  with one exception $8$ (see Table 13).
\begin{table}[h]
\caption{Proper almost universal sums with one exception 8}
\vspace{-3mm}
\begin{center}
\renewcommand{\arraystretch}{1.0}\renewcommand{\tabcolsep}{1mm}
\begin{tabular}{l|l|l}
\hline
Sums &Candidates & Conditions on $\alpha_k$ \\
\hline
$\Delta_{1,1,3,\alpha_4}$ & $3\leq \alpha_4 \leq 17$ & $\alpha_4\neq3^{\star},4^{\star},5^{\star},6^{\star},7^{\star},8^{\star},$\\
&& \hspace{8.5mm}$9,12$\\
\hline
$\Delta_{1,1,3,9,\alpha_5}$ & $9\leq \alpha_5 \leq 17$ & $\alpha_5\neq9,10^{\dagger},11^{\dagger},13^{\dagger},14^{\dagger},15^{\dagger},$\\
&& \hspace{8.5mm}$16^{\dagger},17^{\dagger}$\\
$\Delta_{1,1,3,12,\alpha_5}$ & $12\leq \alpha_5 \leq 89$ & $\alpha_5\neq13^{\dagger},14^{\dagger},15^{\dagger},16^{\dagger},17^{\dagger},81$\\
\hline
$\Delta_{\alpha_1,\alpha_2,\dots,\alpha_{k-1},\alpha_k}(k\geq6)$ & $\alpha_{k-1}\leq\alpha_k\leq\alpha_{k-1}+8$ & $\alpha_k\neq\alpha_{k-1}+\ell$\\
&&$(\ell=0,1^{\dagger},2^{\dagger},3^{\dagger},4^{\dagger},5^{\dagger},6^{\dagger},7^{\dagger},8^{\dagger})$\\
\hline
\end{tabular}
\end{center}
\end{table}

Since the proof of almost universality of each candidate is quite similar to the proof of Theorem \ref{1}, we only provide all parameters for the computations for representations of the ternary quadratic form $f$ (see Table 14).


\begin{table}[h]
\caption{Data for the proof of the candidates when $\small{(\alpha_1,\alpha_2,\alpha_3)=(1,1,3)}$}
\vspace{-3mm}
\begin{center}
\footnotesize
\renewcommand{\arraystretch}{1.1}\renewcommand{\tabcolsep}{0.5mm}
\begin{tabular}{|c|c|c|c|c|c|c|l|}
\hline
$\alpha_4$                     & $\Delta_{\alpha_{i_1},\alpha_{i_2},\alpha_{i_3}}$   &  \multirow{3}{*}{$d$} & \multirow{3}{*}{$a$} & \multirow{3}{*}{$|B|$}       & \multicolumn{2}{c|}{\multirow{2}{*}{$T$}}  & \multirow{3}{*}{$\setlength\arraycolsep{0pt}\begin{array}{l}\text{Sufficient conditions} \\ \text{for } m\rightarrow f\end{array}$}\\ \cline{1-2}
$f$                            &  $h(f)$                                             &                       &                       &                             &  \multicolumn{2}{c|}{}                     &                                                                                                          \\ \cline{1-2} \cline{6-7}
$M_f$                          &  $M_2$                                               &                       &                        &                           &    $z$   & $Q(z)$                          &                         \\
\hline\hline
 every case                    & $\Delta_{1,1,3}$                                   & \multirow{3}{*}{} & \multirow{3}{*}{}& \multirow{3}{*}{} & \multicolumn{2}{|c|}{}  & \multirow{3}{*}{$\setlength\arraycolsep{0pt}\begin{array}{l}m\equiv5\Mod8,\\ m\neq 3^{2u+1}(3v+2)\end{array}$}\\ \cline{1-2}
$(2x+y)^2+y^2+3z^2$             &      $1$                                          &                   &                   &                  &  \multicolumn{2}{c|}{}    &   \\ \cline{1-2}
$\langle2,2,3\rangle$           &                                                   &                   &                    &                  &  \multicolumn{2}{c|}{}      &  \\  
\hline\hline

$\alpha_4=6$    & $\Delta_{1,1,6}$ & \multirow{2}{*}{}       & \multirow{2}{*}{}      & \multirow{2}{*}{}       & \multicolumn{2}{c|}{} & \multirow{3}{*}{$\setlength\arraycolsep{0pt}\begin{array}{l}m\equiv0\Mod8,\\ m\neq3^{2u+1}(3v+1)\end{array}$} \\ \cline{1-2}
$x^2+y^2+6t^2$  &   $1$          &                       &                         &                        &   \multicolumn{2}{c|}{}  &                       \\ \cline{1-2}
$\langle1,1,6\rangle$&            &                         &                         &                     &    \multicolumn{2}{c|}{}   &                      \\
\hline\hline

$\alpha_4=15$                                            & $\Delta_{1,1,15}$      & \multirow{4}{*}{$8$}   & \multirow{4}{*}{$1$}      & \multirow{4}{*}{8}     & \multicolumn{2}{c|}{\multirow{3}{*}{$\begin{pmatrix}8&0&0\\0&7&-5\\0&3&7\end{pmatrix}$}} & \multirow{4}{*}{$\setlength\arraycolsep{0pt}\begin{array}{l} m>9,\\m\equiv1\Mod8,\\ m\neq3^{2u+1}(3v+1)\end{array}$}\\   \cline{1-2}     
$(2x+y)^2+y^2+15t^2$                   &  2                                       &                         &                            &                        & \multicolumn{2}{c|}{}                                                                                                   &\\ \cline{1-2} 
\multirow{2}{*}{$\langle2,2,15\rangle$} & \multirow{2}{*}{$\langle1,6,10\rangle$}  &                        &                         &                             & \multicolumn{2}{c|}{}                                                                                  &\\ \cline{6-7}
                                        &                                         &                         &                                                         &                         & $\pm(1,0,0)$              & $1$                                                                  &\\
\hline
\end{tabular}
\end{center}
\end{table}

\section{Proof of Theorem \ref{5}}

In this section, we give a proof of Theorems \ref{5}.
For positive integers $\alpha_1,\alpha_2,\dots,\alpha_k$, assume that $\Delta_{\alpha_1,\alpha_2,\dots,\alpha_k}$ of triangular numbers is almost universal with the unique exception $2$.
We may assume that $\alpha_1\leq\alpha_2\leq\cdots\leq\alpha_k$.
We know that $\alpha_1$ should be $1$ since it represents 1. 
Then $3\leq\alpha_2\leq4$ since $\mathfrak{T}_2(\Delta_1)=4$.
Note that 
$$
\mathfrak{T}_2(\Delta_{\alpha_1,\alpha_2})=
\begin{cases}
5\quad \text{if}~ (\alpha_1,\alpha_2)=(1,3),\\
8\quad \text{if}~ (\alpha_1,\alpha_2)=(1,4).
\end{cases}
$$
Therefore $(\alpha_1,\alpha_2,\alpha_3)$ should be one of the followings:
$$
(1,3,3),~(1,3,4),~(1,3,5),~(1,4,4),~(1,4,5),~(1,4,6),~(1,4,7),~\text{and}~(1,4,8).
$$
If $(\alpha_1,\alpha_2,\alpha_3)\neq(1,4,5)$, then the second truants are

$$
\mathfrak{T}_2(\Delta_{\alpha_1,\alpha_2,\alpha_3})=
\begin{cases}
{\setlength\arraycolsep{1pt}
\begin{array}{llll}
&5\quad &\text{if}&~(\alpha_1,\alpha_2,\alpha_3)=(1,3,3),\\ 
&11\quad &\text{if}&~(\alpha_1,\alpha_2,\alpha_3)=(1,3,4),\\
&7\quad &\text{if}&~(\alpha_1,\alpha_2,\alpha_3)=(1,3,5),\\
&20\quad &\text{if}&~(\alpha_1,\alpha_2,\alpha_3)=(1,4,4),\\
&8\quad &\text{if}&~(\alpha_1,\alpha_2,\alpha_3)=(1,4,6),\\
&9\quad &\text{if}&~(\alpha_1,\alpha_2,\alpha_3)=(1,4,7),\\
&16\quad &\text{if}&~(\alpha_1,\alpha_2,\alpha_3)=(1,4,8).
\end{array}}
\end{cases}
$$

On the other hand, we checked that $\Delta_{1,4,5}$ represents all nonnegative integers up to $10^7$ except $2$.
We conjectured that the sum $\Delta_{1,4,5}$ of triangular numbers is almost universal with the unique exception 2 (see Conjecture \ref{Conj}).
Note that there is no ternary almost universal sum of triangular numbers with one exception $1,3,4$, and $8$ (see Theorems \ref{1}, \ref{2}, \ref{3}, and \ref{4}).
Therefore, $\Delta_{1,4,5}$ is the unique candidate of ternary almost universal sums of triangular numbers with one exception. 
In this section we assume that Conjecture \ref{Conj} is true.
 
Continuing on with the escalation method, we find all candidates of  34 quaternary and 37 quinary proper almost universal sums of triangular numbers with one exception $2$ (see Table 15).

\begin{table}[hpt!]
\caption{Proper almost universal sums with one exception 2}
\vspace{-3mm}\begin{center}
\renewcommand{\arraystretch}{0.87}\renewcommand{\tabcolsep}{1mm}
\begin{tabular}{l|l|l}
\hline
Sums &Candidates & Conditions on $\alpha_k$ \\
\hline
$\Delta_{1,3,\alpha_3}$ & $3\leq \alpha_3 \leq 5$ & $\alpha_4\neq3,4,5$\\
$\Delta_{1,4,\alpha_3}$ & $4\leq \alpha_3 \leq 8$ & $\alpha_4\neq4,6,7,8$\\
\hline
$\Delta_{1,3,3,\alpha_4}$ & $3\leq \alpha_4 \leq 5$ & $\alpha_4\neq3,4$\\
$\Delta_{1,3,4,\alpha_4}$ & $4\leq \alpha_4 \leq 11$ & $\alpha_4\neq5^{\dagger},9$\\
$\Delta_{1,3,5,\alpha_4}$ & $5\leq \alpha_4 \leq 7$ & $\alpha_4\neq5$\\
$\Delta_{1,4,4,\alpha_4}$ & $4\leq \alpha_4 \leq 20$ & $\alpha_4\neq5^{\dagger},15,18$\\
$\Delta_{1,4,6,\alpha_4}$ & $6\leq \alpha_4 \leq 8$ & $\alpha_4\neq6$\\
$\Delta_{1,4,7,\alpha_4}$ & $7\leq \alpha_4 \leq 9$ & $\alpha_4\neq7$\\
$\Delta_{1,4,8,\alpha_4}$ & $8\leq \alpha_4 \leq 16$ & $\alpha_4\neq14,15$\\
\hline
$\Delta_{1,3,3,3,\alpha_5}$ & $3\leq \alpha_5 \leq 5$ & $\alpha_5\neq3,5^{\dagger}$\\
$\Delta_{1,3,3,4,\alpha_5}$ & $4\leq \alpha_5 \leq 29$ & $\alpha_5\neq4^{\dagger},5^{\dagger},6^{\dagger},7^{\dagger},8^{\dagger},10^{\dagger},$\\
&&\hspace{8.5mm}$11^{\dagger},27$\\
$\Delta_{1,3,4,9,\alpha_5}$ & $9\leq \alpha_5 \leq 11$ & $\alpha_5\neq9,10^{\dagger},11^{\dagger}$\\
$\Delta_{1,3,5,5,\alpha_5}$ & $5\leq \alpha_5 \leq 7$ & $\alpha_5\neq5,6^{\dagger},7^{\dagger}$\\
$\Delta_{1,4,4,15,\alpha_5}$ & $15\leq \alpha_5 \leq 35$ & $\alpha_5\neq16^{\dagger},17^{\dagger},19^{\dagger},20^{\dagger},33$\\
$\Delta_{1,4,4,18,\alpha_5}$ & $18\leq \alpha_5 \leq 20$ & $\alpha_5\neq18,19^{\dagger},20^{\dagger}$\\
$\Delta_{1,4,6,6,\alpha_5}$ & $6\leq \alpha_5 \leq 8$ & $\alpha_5\neq6,7^{\dagger},8^{\dagger}$\\
$\Delta_{1,4,7,7,\alpha_5}$ & $7\leq \alpha_5 \leq 9$ & $\alpha_5\neq7,8^{\dagger},9^{\dagger}$\\
$\Delta_{1,4,8,14,\alpha_5}$ & $14\leq \alpha_5 \leq 16$ & $\alpha_5\neq14,16^{\dagger}$\\
$\Delta_{1,4,8,15,\alpha_5}$ & $15\leq \alpha_5 \leq 17$ & $\alpha_5\neq15,16^{\dagger}$\\
\hline
$\Delta_{\alpha_1,\alpha_2,\dots,\alpha_{k-1},\alpha_k}(k\geq6)$ & $\alpha_{k-1}\leq\alpha_k\leq\alpha_{k-1}+2$ & $\alpha_k\neq\alpha_{k-1}+\ell~(\ell=0,1^{\dagger},2^{\dagger})$\\
\hline
\end{tabular}
\end{center}
\end{table}

Since the proof of almost universality of each candidate is quite similar to the proof of Theorem \ref{1}, we only provide all parameters for the computations for representations of the ternary form $f$ (see Tables 16, 17, 18, 19, 20, 21 and 22).

\begin{table}[hpt!]
\vspace{-0.0mm}
\caption{Data for the proof of the candidates when $\small{(\alpha_1,\alpha_2,\alpha_3)=(1,3,3)}$}
\vspace{-4mm}\begin{center}
\footnotesize
\renewcommand{\arraystretch}{0.95}\renewcommand{\tabcolsep}{0.5mm}
\begin{tabular}{|c|c|l|}
\hline
\multicolumn{2}{|c|}{$\Delta_{\alpha_{i_1},\alpha_{i_2},\alpha_{i_3}}$} & \multirow{2}{*}{Sufficient conditions for $m\rightarrow f$}\\ \cline{1-2}
$f$ & $h(f)$ &                                                                                                                     \\
\hline\hline
\multicolumn{2}{|c|}{$\Delta_{1,3,3}$}&\multirow{2}{*}{$m\equiv7\Mod8,~m\neq 3^{2u}(3v+2)$}\\ \cline{1-2}
$(2x+y)^2+3y^2+3z^2$  & $1$ &                                                                      \\
\hline
\end{tabular}
\end{center}
\end{table}

\begin{table}[h]
\vspace{-3mm}
\caption{Data for the proof of the candidates when $\small{(\alpha_1,\alpha_2,\alpha_3)=(1,3,4)}$}
\vspace{-4mm}\begin{center}
\footnotesize
\renewcommand{\arraystretch}{0.95}\renewcommand{\tabcolsep}{0.5mm}
\begin{tabular}{|c|c|l|}
\hline
$\alpha_4$ & $\Delta_{\alpha_{i_1},\alpha_{i_2},\alpha_{i_3}}$ & \multirow{2}{*}{Sufficient conditions for $m\rightarrow f$}\\ \cline{1-2}
$f$        & $h(f)$                                         &                                                            \\
\hline\hline
$\alpha_4\not\equiv0\Mod3$ & $\Delta_{1,3,4}$ &\multirow{2}{*}{$m\equiv0\Mod8, m\not\equiv0\Mod3$}\\ \cline{1-2}
$(4x+y)^2+3y^2+4z^2$   &   $1$ &  \\  

\hline\hline
$\alpha_4=6$ & $\Delta_{1,3,6}$&\multirow{2}{*}{$m\equiv2\Mod8, m\neq3^{2u}(3v+2)$}\\ \cline{1-2}
$(4x+y)^2+3y^2+6t^2$    &   $1$ &    \\
\hline
\end{tabular}
\end{center}
\end{table}

\begin{table}[h]
\vspace{-3mm}
\caption{Data for the proof of the candidates when $\small{(\alpha_1,\alpha_2,\alpha_3)=(1,3,5)}$}
\vspace{-4mm}\begin{center}
\footnotesize
\renewcommand{\arraystretch}{1.1}\renewcommand{\tabcolsep}{0.5mm}
\begin{tabular}{|c|c|c|c|c|c|c|l|}
\hline
\multicolumn{2}{|c|}{$\Delta_{\alpha_{i_1},\alpha_{i_2},\alpha_{i_3}}$}             &  \multirow{3}{*}{$d$} & \multirow{3}{*}{$a$} & \multirow{3}{*}{$|B|$}       & \multicolumn{2}{c|}{\multirow{2}{*}{$T$}}  & \multirow{3}{*}{$\setlength\arraycolsep{0pt}\begin{array}{l}\text{Sufficient conditions} \\ \text{for } m\rightarrow f\end{array}$}\\ \cline{1-2}
$f$                            &  $h(f)$                                             &                       &                       &                             &  \multicolumn{2}{c|}{}                     &                                                                                                          \\ \cline{1-2} \cline{6-7}
$M_f$                          &  $M_2$                                              &                      &                        &                            &  $z$            &  $Q(z)$                  &                                                                                                          \\
\hline  \hline

 \multicolumn{2}{|c|}{$\Delta_{1,3,5}$}                                                                                            & \multirow{4}{*}{$8$}   & \multirow{4}{*}{$1$}                      & \multirow{4}{*}{4}      & \multicolumn{2}{c|}{\multirow{3}{*}{$\setlength\arraycolsep{2pt}\begin{pmatrix}8&0&0\\0&8&4\\0&-4&6\end{pmatrix}$}} & \multirow{4}{*}{$\setlength\arraycolsep{0pt}\begin{array}{l}m>1,\\ m\equiv1\Mod8,\\m\notequiv0\Mod5\end{array}$}\\ \cline{1-2} 
{$(2x+y)^2+3y^2+5z^2$}   & 2                                                                                                                       &                      &                                            &                       & \multicolumn{2}{c|}{}                                                                                                 &\\
\hline   
\end{tabular}
\end{center}
\end{table}

\begin{table}[h]
\vspace{-3mm}
\caption{Data for the proof of the candidates when $\small{(\alpha_1,\alpha_2,\alpha_3)=(1,4,4)}$}
\vspace{-4mm}\begin{center}
\footnotesize
\renewcommand{\arraystretch}{1.1}\renewcommand{\tabcolsep}{0.5mm}
\begin{tabular}{|c|c|c|c|c|l|}

\hline
$\alpha_4$                     & $\Delta_{\alpha_{i_1},\alpha_{i_2},\alpha_{i_3}}$   &  \multirow{3}{*}{$d$} & \multirow{3}{*}{$a$} & \multirow{3}{*}{$|B|$}        & \multirow{3}{*}{$\setlength\arraycolsep{0pt}\begin{array}{l}\text{Sufficient conditions} \\ \text{for } m\rightarrow f\end{array}$}\\ \cline{1-2}
$f$                            &  $h(f)$                                             &                       &                       &                               &                                                                                                          \\ \cline{1-2} 
$M_f$                          &  $M_2$                                               &                       &                        &                              &                         \\
\hline\hline

$\alpha_4\not\equiv0\Mod3$                                                   & $\Delta_{1,4,4}$                       & \multirow{4}{*}{$3$}   & \multirow{4}{*}{$\setlength\arraycolsep{0pt}\begin{array}{c}0,2\end{array}$}  & \multirow{4}{*}{0}     & \multirow{4}{*}{$\setlength\arraycolsep{0pt}\begin{array}{l}m\equiv1\Mod8,\\ m\equiv0,2\Mod3\end{array}$}\\   \cline{1-2}     
$(2x+y)^2+4y^2+4z^2$                                                        &  2                                       &                         &                                                                             &                          &\\ \cline{1-2} 
\multirow{2}{*}{$\begin{pmatrix}4&2\\2&5\end{pmatrix}\perp\langle4\rangle$} & \multirow{2}{*}{$\langle1,4,16\rangle$}  &                        &                                                                            &                            &\\ 
                                                                             &                                         &                         &                                                                            &                           &\\
\hline

\end{tabular}
\end{center}
\end{table}


\begin{table}[hbt!]
\vspace{-3mm}
\caption{Data for the proof of the candidates when $\small{(\alpha_1,\alpha_2,\alpha_3)=(1,4,6)}$}
\vspace{-4mm}\begin{center}
\footnotesize
\renewcommand{\arraystretch}{0.95}\renewcommand{\tabcolsep}{0.5mm}
\begin{tabular}{|c|c|l|}
\hline
\multicolumn{2}{|c|}{$\Delta_{\alpha_{i_1},\alpha_{i_2},\alpha_{i_3}}$} & \multirow{2}{*}{Sufficient conditions for $m\rightarrow f$}\\ \cline{1-2}
$f$ & $h(f)$ &                                                                                                                     \\
\hline\hline
\multicolumn{2}{|c|}{$\Delta_{1,4,6}$}&\multirow{2}{*}{$m\equiv3\Mod8,~m\not\equiv0\Mod3$}\\ \cline{1-2}
$x^2+4y^2+6z^2$  & $1$ & \\
\hline
\end{tabular}
\end{center}
\end{table}


\begin{table}[hbt!]
\vspace{-1mm}
\caption{Data for the proof of the candidates when $\small{(\alpha_1,\alpha_2,\alpha_3)=(1,4,7)}$}
\vspace{-4mm}\begin{center}
\footnotesize
\renewcommand{\arraystretch}{0.91}\renewcommand{\tabcolsep}{0.5mm}
\begin{tabular}{|c|c|c|c|c|c|c|l|}
\hline
$\alpha_4$ & $\Delta_{\alpha_{i_1},\alpha_{i_2},\alpha_{i_3}}$             &  \multirow{3}{*}{$d$} & \multirow{3}{*}{$a$} & \multirow{3}{*}{$|B|$}       & \multicolumn{2}{c|}{\multirow{2}{*}{$T$}}  & \multirow{3}{*}{$\setlength\arraycolsep{0pt}\begin{array}{l}\text{Sufficient conditions} \\ \text{for } m\rightarrow f\end{array}$}\\ \cline{1-2}
$f$                            &  $h(f)$                                             &                       &                       &                             &  \multicolumn{2}{c|}{}                     &                                                                                                          \\ \cline{1-2} \cline{6-7}
$M_f$                          &  $M_2$                                              &                      &                        &                            &  $z$            &  $Q(z)$                  &                                                                                                          \\
\hline\hline

$\alpha_4=8$                                                       &  $\Delta_{1,4,8}$                       & \multirow{4}{*}{$3$}   & \multirow{4}{*}{$1$}      & \multirow{4}{*}{2}     & \multicolumn{2}{c|}{\multirow{3}{*}{$\begin{pmatrix}3&0&0\\0&1&-8\\0&1&1\end{pmatrix}$}} & \multirow{9}{*}{$\setlength\arraycolsep{0pt}\begin{array}{l}$m>1$,\\ m\equiv5\Mod8,\\m\equiv1\Mod3 \text{ or }\\ m\equiv2,3\Mod5\text{ or }\\ m\equiv0\Mod{15}   \end{array}$}\\  \cline{1-2}
$(2x+t)^2+4y^2+8t^2$                                               & $2$                                     &                        &                           &                        & \multicolumn{2}{c|}{}                                                                                                   &\\
\hline
\end{tabular}
\end{center}
\end{table}


\begin{table}[hbt!]
\caption{Data for the proof of the candidates when $\small{(\alpha_1,\alpha_2,\alpha_3)=(1,4,8)}$}
\vspace{-4mm}\begin{center}
\footnotesize
\renewcommand{\arraystretch}{0.91}\renewcommand{\tabcolsep}{0.2mm}
\begin{tabular}{|c|c|c|c|c|c|c|l|}
\hline
$\alpha_4$ & $\Delta_{\alpha_{i_1},\alpha_{i_2},\alpha_{i_3}}$             &  \multirow{3}{*}{$d$} & \multirow{3}{*}{$a$} & \multirow{3}{*}{$|B|$}       & \multicolumn{2}{c|}{\multirow{2}{*}{$T$}}  & \multirow{3}{*}{$\setlength\arraycolsep{0pt}\begin{array}{l}\text{Sufficient conditions} \\ \text{for } m\rightarrow f\end{array}$}\\ \cline{1-2}
$f$                            &  $h(f)$                                             &                       &                       &                             &  \multicolumn{2}{c|}{}                     &                                                                                                          \\ \cline{1-2} \cline{6-7}
$M_f$                          &  $M_2$                                              &                      &                        &                            &  $z$            &  $Q(z)$                  &                                                                                                          \\
\hline\hline

$\alpha_4\neq9,10,12$                                & $\Delta_{1,4,8}$               & \multirow{4}{*}{$3$}   & \multirow{4}{*}{$1$}      & \multirow{4}{*}{2}     & \multicolumn{2}{c|}{\multirow{3}{*}{$\begin{pmatrix}3&0&0\\0&1&-8\\0&1&1\end{pmatrix}$}} & \multirow{9}{*}{$\setlength\arraycolsep{0pt}\begin{array}{l}$m>1$,\\ m\equiv5\Mod8,\\m\equiv1\Mod3 \text{ or }\\ m\equiv2,3\Mod5\text{ or }\\ m\equiv0\Mod{15}   \end{array}$}\\  \cline{1-2}
 $(2x+z)^2+4y^2+8z^2$                                              &   $2$                                 &                        &                           &                        & \multicolumn{2}{c|}{}                                                                                                   &\\
\hline

\end{tabular}
\end{center}
\end{table}



\begin{thebibliography}{abcd}

\bibitem {B} M. Bhargava, 
{\em On the Conway-Schneeberger fifteen theorem}, 
Contem. Math. \textbf{272}(2000), 27--38. 

\bibitem{BH} M. Bhargava and J. Hanke,
{\em Universal quadratic forms and the 290 theorem},
preprint.

\bibitem{BO} J. Bochnak and B.-K. Oh,
{\em Almost universal quadratic forms: An effective solution of a problem of Ramanujan},
Duke Math. J. \textbf{147}(2009), 131-156.

\bibitem{BK} W. Bosma and B. Kane,
{\em The triangular theorem of eight and representation by quadratic polynomials},
Proc. Amer. Math. Soc. \textbf{141}(2013), 1473–1486.

\bibitem{dick} L. E. Dickson,
{\em Quaternary quadratic forms representing all integers},
Amer. J. Math.  \textbf{49}(1927), 39-56.

\bibitem{hal} P. R. Halmos,
{\em Note on almost-universal forms},
Bull. Amer. Math. Soc. \textbf{44}(1938), 141-144.

\bibitem{Kaps} J. Ju,
{\em Ternary quadratic forms representing the same integers},
to appear at Int. J. Number Theory.


\bibitem{4-8} J. Ju,
{\em Universal mixed sums of generalized $4$- and $8$-gonal numbers},
Int. J. Number Theory \textbf{16}(2020) 603-627.

\bibitem{ternary} J. Ju, B.-K. Oh, B. Seo
{\em Ternary universal sums of generalized polygonal numbers}
Int. J. Number Theory \textbf{15}(2019), 655-675.

\bibitem{Kane} B. Kane,
{\em Representing sets with sums of triangular numbers},
Int. Math. Res. Not. \textbf{2009}(2009), 3264–3285.

\bibitem {regular} B.-K. Oh, 
{\em Regular positive ternary quadratic forms}, 
Acta Arith. \textbf{147}(2011), 233-243. 

\bibitem {pentagonal}  B.-K. Oh, 
{\em Ternary universal sums of generalized pentagonal numbers},  
J. Korean Math. Soc. \textbf{48}(2011), 837-847.

\bibitem{om} O. T. O'Meara, 
{\em Introduction to quadratic forms},
Springer Verlag, New York, 1963.

\bibitem{pall} G. Pall,
{\em An almost universal form},
Bull. Amer. Math. Soc. \textbf{46}(1940), 291. 

\bibitem{rama} S. Ramanujan, 
{\em On the expression of a number in the form $ax^2+by^2+cz^2+du^2$},
Proc. Camb. Phil. Soc. \textbf{19} (1916), 11-21.

\bibitem {sp2} R. Schulze-Pillot, 
{\em Darstellung durch spinorgeschlechter tern\"arer quadratischer formen}, 
J. Number Theory $\mathbf{12}$(1980), 529--540.
\end{thebibliography}
\end{document}